\documentclass[11pt, reqno]{amsart}
\usepackage[margin=1.4in]{geometry}
\usepackage{mathtools,amsmath, amsthm, amsfonts, amssymb, latexsym, amsthm,
  amscd, cancel, enumerate, rotating, comment, appendix, 
  makecell, paralist,times,graphicx,pgf,dirtytalk,
appendix,float}
\usepackage[colorlinks=true, pdfstartview=FitV, linkcolor=blue, citecolor=blue, urlcolor=blue]{hyperref}
\usepackage[capitalise, noabbrev]{cleveref}
\usepackage{mathrsfs}
\usepackage{paralist}
\usepackage{pgf}
\usepackage{times}
\usepackage{enumitem}
\usepackage{caption}
\usepackage{subcaption}
\usepackage[T1]{fontenc} 
\usepackage{url}

\numberwithin{equation}{section}

\newtheorem{theorem}{Theorem}[section]
\newtheorem{introtheorem}{Theorem}

\newtheorem*{introtheorem*}{Main Theorem}

\newtheorem{corollary}[theorem]{Corollary}

\newtheorem{lemma}[theorem]{Lemma}
\newtheorem{proposition}[theorem]{Proposition}

\theoremstyle{definition}
\newtheorem{definition}[theorem]{Definition}
\newtheorem*{defn*}{Definition}

\newtheorem{example}[theorem]{Example}

\newtheorem{question}[theorem]{Question}

\newtheorem*{ack}{Acknowledgements}
\newtheorem{remark}[theorem]{Remark}

\numberwithin{equation}{section}

\newcommand{\lcm}{\mathrm{lcm}}
\newcommand{\LCM}{\mathrm{LCM}}
\newcommand{\Tor}{\mathrm{Tor}}
\newcommand{\im}{\mathrm{im}}

\newcommand{\supp}{\mathrm{Supp}}

\newcommand{\Ass}{\mathrm{Ass}}

\newcommand{\ba}{\mathbf{a}}
\newcommand{\bb}{\mathbf{b}}

\newcommand{\bc}{\mathbf{c}}

\newcommand{\be}{\mathbf{e}}

\newcommand{\bj}{\mathbf{1}}
\newcommand{\bu}{\mathbf{u}}

\newcommand{\bz}{\mathbf{z}}
\newcommand{\bx}{\mathbf{x}}

\newcommand{\bm}{\mathbf{m}}

\newcommand{\by}{\mathbf{y}}

\newcommand{\m}{\mathbf{m}}

\newcommand{\bma}{\m^\ba}
\newcommand{\fp}{\mathfrak{p}}
\newcommand{\fq}{\mathfrak{q}}

\newcommand{\sfk}{\mathsf k}

\DeclareMathOperator{\conv}{conv}

\newcommand{\C}{\mathcal{C}}
\newcommand{\DD}{\mathcal{D}}
\newcommand{\E}{\mathcal{E}}

\newcommand{\N}{\mathcal{N}}

\newcommand{\J}{\mathcal{J}}
\newcommand{\K}{\mathcal{K}}

\newcommand{\Ll}{\mathcal{L}}

\newcommand{\NN}{\mathbb{N}}
\newcommand{\RR}{\mathbb{R}}

\newcommand{\Pp}{\mathcal{S}}
\newcommand{\CC}{\mathscr{C}}

\newcommand{\unsure}{\theta}

\newcommand{\ssm}{\smallsetminus}

\newcommand{\Erq}{{\E_q}^r}

\newcommand{\SqI}{S_{[q],I}} 

\newcommand{\Nrq}{\N^r_q}

\newcommand{\e}{\epsilon}
\newcommand{\pme}{{\pmb{\e}}}
\newcommand{\pmea}{\pme^{\ba}}

\newcommand{\qand}{\quad \mbox{and} \quad }
\newcommand{\qor}{\quad \mbox{or} \quad }
\newcommand{\qif}{\quad \mbox{if} \quad }
\newcommand{\qfor}{\quad \mbox{for} \quad }
\newcommand{\qwhere}{\quad \mbox{where} \quad }
\newcommand{\qwith}{\quad \mbox{with} \quad }
\newcommand{\qwhen}{\quad \mbox{when} \quad }

\newcommand{\qforall}{\quad \mbox{for all} \quad }

\newcommand{\compl}[1]{#1^{\sf{c}}}

\newcommand{\st}{\colon}

\newcommand{\divides}{\mathrel{|}}

\DeclareMathOperator{\ass}{Ass}

\DeclareMathOperator{\minimalprimes}{Min}
\DeclareMathOperator{\mingens}{Mingens}
\DeclareMathOperator{\sdefect}{sdefect}
\DeclareMathOperator{\idefect}{idefect}

\newcommand{\bfemph}[1]{\textbf{#1}}
\renewcommand{\emph}[1]{\bfemph{#1}} 

\title[Symbolic powers and integral closures]{ 
Symbolic powers and integral closures via extremal ideals}

\author{Trung Chau}
\author{Art Duval}
\author{Sara Faridi}
\author{Thiago Holleben}
\author{Susan Morey}
\author{Liana \c Sega}

\definecolor{DarkGreen}{RGB}{50,203,0}


\begin{document}
\begin{abstract} 
 This paper demonstrates that extremal ideals can be used to great effect to compute integral closures of powers and symbolic powers of square-free monomial ideals. We show that the generators of these powers are images of the generators of the corresponding powers of extremal ideals under a specific ring homomorphism.
 Extremal ideals provide sharp bounds for a variety of invariants widely studied in the literature, including resurgence, asymptotic resurgence, and symbolic defect, as well as Betti numbers of symbolic powers and of integral closures of powers of square-free monomial ideals. 
When restricted to the class of extremal ideals, algebraic computations are reduced to problems of discrete geometry and linear programming, allowing the use of a wide variety of techniques.
 As a result, in situations where computations are feasible for extremal ideals, we provide concrete sharp bounds for many of these invariants. Our methods reduce finding homological invariants and algebraic constructions for infinitely many ideals to computations for a single highly symmetric ideal, based solely on the number of generators.

\end{abstract} 
\maketitle



\section{\bf Introduction }
The behavior of powers of ideals has been a subject of great
scrutiny and multiple conjectures over the past century. There are
often ways to describe this behavior asymptotically, or for certain
classes of ideals, but in general there is much we still do not understand. A common theme with such problems is that powers of square-free monomial ideals in a polynomial ring $R$ can exhibit unexpected properties.
This paper offers a way to study symbolic powers $I^{(r)}$ and integral closures of powers $\overline{I^r}$ of {\it all} square-free monomial ideals $I$ with $q$ generators via a single ideal, namely the {\it extremal ideal} $\E_q$ (\cref{d:extremal}) introduced in \cite{Lr}. 
The principle is that all the complications faced when dealing with powers of square-free monomial ideals are concentrated in the extremal ideals $\E_q$. The  ring homomorphism $\psi_I$ (\cref{d:psi}), which is known to satisfy $\psi_I(\Erq) = I^rR$,  then offers a way to study any $I$ via extremal ideals.

\begin{introtheorem}
[\cref{thm:integral-closure-extremal,thm:symbolic-extremal}]\label[theorem]{t:intro}
If $I$  is an ideal minimally generated by $q \ge 1$ square-free monomials in a polynomial ring $R$, and $r$ is a positive integer, then
$$
\overline{I^r}= \psi_I(\overline{\Erq})R
\qand
I^{(r)}=\psi_I({\E_q}^{(r)})R.
$$
\end{introtheorem}

In other words,  $\psi_I$ {\it preserves} integral closures of powers $\overline{I^r}$ and symbolic powers $I^{(r)}$ in the sense that $\psi_I$ maps generators onto the generating set of the corresponding powers of $I$.
Although there are algorithms to compute the monomial ideals  $\overline{I^r}$ and $I^{(r)}$, the complexity skyrockets when $r$ increases. See for example \cite{integral-closure-algorithm} for a history, \cite{Vasconcelos-survey} for a more theoretical analysis, and \cite{symbolic-survey} for a survey of problems that highlight the difficulty of the computations. 

 In view of \cref{t:intro}, to 
 compute $\overline{I^{r}}$ and $I^{(r)}$ for any square-free
 monomial ideal $I$ with $q$ generators, it suffices to do so in
 the case $I=\E_q$ and then apply the map $\psi_I$.
In this paper, we study both $\overline{\Erq}$ and ${\E_q}^{(r)}$, and then examine the implications of our results for any square-free monomial ideal $I$ generated by $q$ elements.
 
\subsubsection*{\bf Integral closure:}  Given an ideal $I$ in a ring $R$, an element $r\in R$ is called \emph{integral over $I$} if there exists a polynomial $f(x) = x^t+a_{1}x^{t-1}+\cdots + a_t=0$ with $a_i \in I^i$ such that $f(r)=0$. The set of all integral elements of $I$ is called the \emph{integral closure} of $I$ (in $R$), denoted by $\overline{I}$. Integral closure has played an important role in number theory and algebraic geometry since the work of Krull in the 1930s \cite{Krull1930,Krull1932}. Specifically, the integral closure of an ideal is used in the study of singularities, Rees algebras, blow ups, and valuations. 

The integral closure of the $r$-th power $\overline{I^r}$ is of great interest to commutative algebraists. As with ordinary powers $I^r$, integral closures of powers also exhibit asymptotic behaviors. When $I$ is a monomial ideal in a polynomial ring $R$ over a field, it is known that the integral closure $\overline{I}$ is also a monomial ideal, and it contains monomials $m$ such that $m^t\in I^t$, for some integer $t$. The problem of finding the integral closure in this case can be translated into the language of integer programming or convex polytopes. In fact, in many programming languages, e.g., Normaliz~\cite{Normaliz} or Macaulay2~\cite{M2}, the integral closure is computed using convex geometry, not commutative algebra. 

In this paper we use the fact that $\overline{I^r}$ can be calculated via an application of $\psi_I$ to $\overline{\Erq}$ by \cref{t:intro}. This point of view allows us to describe a \say{test element} $g(r,q)$ (see \cref{e:newfrq}) which can be used to test the equality $\overline{I^r}=I^r$. For low $r$ and $q$, using $g(r,q)$ proves particularly effective (see \cref{t:general-I-ic}\eqref{i:g(2,4)}). Using this approach, we show the following.

\begin{introtheorem}[\cref{c:integralclosurepsi,t:general-I-ic,thm:integral-defect}]\label{t:intro-bounds-integral-closure}
 If $I$  is an ideal minimally generated by $q \ge 1$ square-free monomials  $m_1,\ldots,m_q$ in a polynomial ring $R$, and $r$ is a positive integer, then 
    \begin{enumerate}
        \item   $\overline{I^r}=I^r+        
    \psi_I(\overline{\Erq}\smallsetminus  \Erq)R$;
    \item  $\overline{I^r}\ne I^r$ if  $\psi_I(g(r,q))\notin I^r$ for $r \ge 2$ and $q \ge 4$; 
       \item $\overline{I^r}=I^r$ if $q \leq 3$;
 \item if $q=4$, then the following are equivalent:
 \begin{enumerate}
 \item $\overline{I^s}=I^s$ for all $s\ge 1$;
 \item $\overline{I^2}=I^2$;  
 \item $\gcd(m_a,m_b)\mid \lcm(m_c,m_d)$ for some $\{a,b,c,d\}=\{1,2,3,4\}$.
\end{enumerate}
\item  if $q =5$, then the following are equivalent:
\begin{enumerate}
\item $\overline{I^s}=I^s$ for all $s\ge 1$;
\item $\overline{I^2}=I^2$ and  $\overline{I^{3}}=I^3.$
\end{enumerate}
\end{enumerate}
\end{introtheorem}

\subsubsection*{\bf Symbolic powers:} 
The \emph{$r$-th symbolic power} of an ideal $I$ (of a polynomial ring $R$) is defined as 
\[
I^{(r)} =\bigcap_{P\in \minimalprimes(I)} (I^r R_P \cap R).
\]
For a prime ideal $P$ (in a polynomial ring over a field) and an integer $r$, the $r$-th symbolic power of $P$ consists of polynomials that vanish to order at least $r$ on the variety defined by $P$. In general, $I^{r}\subseteq I^{(r)}$, but equality does not always hold. Symbolic powers are often studied in geometry, where $\Ass(I)$ is used in place of $\minimalprimes(I)$ in the definition; however, $\Ass(I)=\minimalprimes(I)$ for square-free monomial ideals and so the two definitions of symbolic powers converge in our setting.  Symbolic powers directly connect commutative algebra and algebraic geometry. They have appeared as a fundamental tool in birational geometry, singularity theory, and intersection theory. 

In commutative algebra, symbolic powers form a topic that is interesting and useful in its own right. They have appeared as auxiliary tools in several important results in commutative
algebra, including Krull's Principal Ideal Theorem and Chevalley's Lemma \cite{HKV1995,Milne2020PrimerCA}. Symbolic powers of monomial ideals can be translated to the language of integer programming, as with their integral closure counterparts. 

We obtain an analog for \cref{t:intro-bounds-integral-closure} for symbolic powers of square-free monomial ideals.  We show that for each $q$, there is a test element $h(r,q)$, defined in \cref{p:not normally torsion free}, that can identify key properties of symbolic powers. Similar to integral closures of powers, the \say{test element} $h(r,q)$ proves particularly effective for low $q$ and $r$ (see \cref{ex:p3}). We provide a complete explicit description of all minimal generators of ${\E_q}^{(2)}$ for any $q$, and consequently obtain a set of generators for $I^{(2)}$ for any square-free monomial ideal $I$. 

\begin{introtheorem}[\cref{c:symbolic=ordinary,c:E2-symbolic,t:general-I-symbolic}]\label{t:intro-bounds-symbolic-powers}
  If $I$  is an ideal minimally generated by $q \ge 1$ square-free monomials in a polynomial ring $R$, and $r$ is a positive integer, then
    \begin{enumerate}
        \item   $I^{(r)}=I^r+        
    \psi_I({\E_q}^{(r)}\smallsetminus  \Erq)R$;
                \item $I^{(r)}\ne I^r$ if $\psi_I(h(r,q))\notin I^r$ for $r\geq 2$ and $q\geq 3$;
    \item $I^{(r)}=I^r$ for  $q\leq 2$;
    \item $I^{(2)}= \big( \psi_I(f_X  ) \colon \emptyset \neq X\subseteq [q] \big )$, where $f_X$ is as in \cref{eq:f-defn}.
\end{enumerate}
\end{introtheorem}

The fact that $\psi_I$ preserves integral closures of powers and symbolic powers gives upper bounds on a multitude of their homological and algebraic invariants. 
One of the main results in~\cite{Lr} is an inequality between Betti numbers
$$
    \beta_i(I^r) \leq \beta_i({\E_q}^r) \qforall i,r \ge 1,
$$
where $I$ is an arbitrary square-free monomial ideal on $q$ generators. In particular, once $q$ and $r$ are fixed, $\Erq$ provides the maximum $i$-th Betti number for any $i$ among square-free monomial ideals with $q$ generators, justifying the choice of word \say{extremal}. In this paper, 
we show that $\E_q$ exhibits extremal  behavior with respect to other well known algebraic and numerical invariants, which are listed in \cref{t:intro-bounds}. Besides Betti numbers, \cref{t:intro-bounds} includes well-known numerical invariants related to measuring the difference between ordinary powers and symbolic powers such as the symbolic defect $\sdefect(r,I)$, resurgence $\rho(I)$, and asymptotic resurgence $\rho_a(I)$, as well as the integral defect $\idefect(r,I)$, which we define in this paper as the integral closure counterpart of symbolic defect.

\begin{introtheorem}[\cref{t:upperbound,thm:integral-defect,c:sdefectbound,c:rho,t:general-I-symbolic,p:resurgence-cases}]
\label{t:intro-bounds}
    If $I$  is an ideal minimally generated by $q \ge 1$ square-free monomials in a polynomial ring $R$, and $r$ is a positive integer, then
    \begin{enumerate}
        \item $\beta_i(\overline{I^r}) \leq \beta_i(\overline{\Erq}) \qand \beta_i(I^{(r)}) \leq \beta_i({\E_q}^{(r)})$  for any integer $i$;
        \item $\idefect(r, I) \leq \idefect(r, \E_q)\qand \sdefect(r, I) \leq \sdefect(r, \E_q)$;
        \item  $\sdefect(2,\E_q) = 2^q -1-q-\binom{q}{2}$;
        \item $\rho(I) \leq \rho(\E_q) \qand \rho_a(I) \leq \rho_a(\E_q)$;
        \item $\rho(\E_3)=\rho_a(\E_3)=\frac{4}{3}$,\, $\rho_a(\E_4)=\frac{3}{2}$ \,and\, $\rho_a(\E_5)=\frac{8}{5}$.
    \end{enumerate}
  \end{introtheorem}

All of these results show that computing invariants and bounds on these invariants for powers of arbitrary monomial ideals reduces to, or is greatly aided by, computing those invariants just once, for powers of the extremal ideal. These invariants are in general difficult to compute.  We will see in \cref{sec:integral-closure,sec:2nd-symbolic-powers} that for extremal ideals, these computations can be completely described in terms of linear programming.
However, while the number of generators of the extremal ideal is $q$, the number of variables in the extremal ideal is very large, growing exponentially in the number of generators.  As a result, even the linear programming problems grow quickly in $q$ and $r$.  Still, the universality of the extremal ideal computations make them worthwhile.

This article is structured as follows. In \cref{sec:preliminaries} we recall the definitions of extremal ideals and of the ring map $\psi_I$ and prove basic results about these ideals and maps. \cref{sec:integral-closure} focuses on integral closures of powers of extremal ideals. We prove the first main result of the paper, \cref{thm:integral-closure-extremal}, stating that $\psi_I$ preserves integral closures of powers. We introduce the integral defect, an invariant measuring how far $\overline{I^r}$ is from $I^r$, which is the integral closure counterpart of the symbolic defect. \cref{sec:symbolic} is a parallel study of symbolic powers, and \cref{sec:2nd-symbolic-powers} is dedicated to the second symbolic power. In particular, in \cref{c:explicit-bound} we provide an explicit bound dependent only on $q$ for the second symbolic defect of any square-free monomial ideal with $q$ generators. Finally, \cref{sec:betti-numbers} consists of results showing how free resolutions of $\overline{I^r}$ or $I^{(r)}$ interact with the map $\psi_I$, and how extremal ideals bound their Betti numbers.

\section{\bf  The map $\psi_I$ and extremal ideals} \label[section]{sec:preliminaries}

It is known that every square-free monomial ideal $I$  can be realized as the image of an extremal ideal under a ring homomorphism $\psi_I$ (\cite{Lr}). The premise of this paper and its sequel~\cite{AlgII} is that
the algebraic behavior of all powers of $I$ can be traced back to the powers of the extremal ideal via $\psi_I$.  
We begin by recalling definitions and basic properties related to extremal ideals and the ring homomorphism $\psi_I$.

\begin{definition}[{\bf Extremal ideals}]\label[definition]{d:extremal}
Let $q$ be a positive integer and consider the
polynomial ring
$$S_{[q]}=\sfk\big [ y_A \st \emptyset \neq A \subseteq [q]\big ]$$
over a field $\sfk$.  For each $i \in [q]$ define a square-free
monomial $\e_i$ in $S_{[q]}$ as
  $$\e_i= \prod_{\tiny \substack{\emptyset\ne A \subseteq [q]\\ i \in A}} y_A.$$ The square-free monomial ideal $\E_q =
  (\e_1,\ldots, \e_q)$ is called the {\bf $\pmb{q}$-extremal ideal}.
  \end{definition}

We will usually simplify the
notation by writing $y_{_1}$ for $y_{\{1\}}$, $y_{_{12}}$ for $y_{\{1,2\}}$,
etc., and refer to a $q$-extremal ideal simply as an extremal ideal.  Although we do not consider $y_{\emptyset}$ in general, when $y_{\emptyset}$ occurs due to set operations, we will use the convention that $y_{\emptyset}=1$.

\begin{example} When $q=4$, the ideal $\E_4$ is generated by the  monomials
  $$\begin{array}{ll} \e_1&=y_{_1} y_{_{12}} y_{_{13}} y_{_{14}} y_{_{123}} y_{_{124}}
    y_{_{134}}  y_{_{1234}}  ; \\ \e_2&=y_{_2} y_{_{12}} y_{_{23}}
    y_{_{24}} y_{_{123}} y_{_{124}} y_{_{234}} y_{_{1234}}
    ;\\ \e_3&=y_{_3} y_{_{13}} y_{_{23}} y_{_{34}} y_{_{123}} y_{_{134}} y_{_{234}}
    y_{_{1234}} ;\\ \e_4&=y_{_4} y_{_{14}} y_{_{24}} y_{_{34}}
    y_{_{124}} y_{_{134}} y_{_{234}} y_{_{1234}} \\
  \end{array}$$
in the polynomial ring $\sfk[y_{_1}, y_{_2}, y_{_3}, y_{_4}, y_{_{12}}, y_{_{13}}, y_{_{14}}, y_{_{23}}, y_{_{24}}, y_{_{34}}, y_{_{123}}, y_{_{124}}, y_{_{134}}, y_{_{234}}, y_{_{1234}}]$. 
\end{example}

Throughout this paper, unless stated otherwise, $R$ denotes the polynomial ring $\sfk[x_1,\dots, x_n]$ where $\sfk$ is a field, and $I$ is a square-free monomial ideal in $R$ minimally generated by monomials $m_1,\ldots,m_q$, with $q \ge 1$. The {\bf support} of a monomial $m$ is the set 
\[
\supp(m)=\{ x_k \st k \in [n] \qand x_k\mid m\}.
\]

We introduce a function $\unsure_I$ which keeps track of clusters of variables that always appear together in the generators of $I$.
For each  $\emptyset\ne A\subseteq [q]$, set
\begin{equation}\label{eq:gamma}
\unsure_I(A)=\bigcap_{j\in A}\supp(m_j)\smallsetminus  \bigcup_{j\in [q]\smallsetminus  A}\supp(m_j).
\end{equation}
Note that by the definition, the sets $\unsure_I(A)$ for $\emptyset \ne A \subseteq [q]$ that are nonempty form a partition of the set of variables appearing in the $m_i$, namely $\cup_{i=1}^q \supp(m_i)$.

\begin{example}\label[example]{e:psi} 
   Let $I$ be the ideal in $\sfk[x_1, \ldots , x_7]$ generated $q=3$ 
   square-free monomials
   $$m_1 = x_1x_2x_5x_7, \quad m_2 = x_2x_3x_7, \quad m_3 = x_3x_4x_6.$$ 
   Clustering the variables that always appear together, we rewrite
    $$m_1 = (x_2x_7)(x_1x_5), \quad 
    m_2 = (x_2x_7)(x_3), \quad 
    m_3 = (x_3)(x_4x_6).$$ 
   Then 
   $$\unsure_I(\{1,2\})=\{x_2,x_7\}, \quad 
   \unsure_I(\{1\})=\{x_1,x_5\}, \quad 
   \unsure_I(\{2,3\})=\{x_3\}, \quad 
   \unsure_I(\{3\})=\{x_4,x_6\}$$
   and $\unsure_I(A)=\emptyset$ for all other $A$.
   \end{example}

For convenience, we give an equivalent description of $\unsure_I$, which is straightforward to observe from \cref{eq:gamma}.

\begin{lemma}\label[lemma]{l:gamma}
If I is an ideal minimally generated by square-free monomials $m_1,\ldots,m_q$ in the polynomial ring $R=\sfk[x_1,\ldots,x_n]$, and $\emptyset \neq A \subseteq [q]$, then $\unsure_I(A)$ is the set consisting of all $x_k$ with $k\in [n]$ such that   
$$x_k \mid m_j\iff j\in A.$$

\end{lemma}

\begin{definition}[{\bf The ring homomorphism $\psi_I$}]\label[definition]{d:psi}
  Let $I$ be an ideal minimally generated by
  square-free monomials $m_1,\ldots,m_q$ of the 
  polynomial ring $R=\sfk[x_1,\ldots,x_n]$. 
  Define $\psi_I$ to be the
  ring homomorphism $$\psi_I \colon S_{[q]} \to R \qwhere
  \psi_I(y_A)=\prod_{x_k\in \unsure_I(A)}x_k
 $$
where we use the convention is that $\psi_I(y_A)=1$ when $\unsure_I(A)=\emptyset$.  Alternatively, $\psi_I(y_A)$ can be thought as the square-free monomial with support equal to $\unsure_I(A)$.
\end{definition}

\begin{example} If $I$ is as in \cref{e:psi}, then    
$\psi_I: \sfk\big{[}y_A\colon \emptyset\ne A\subseteq \{1,2,3\}\big{]} \to \sfk[x_1,
  \ldots, x_7]$ is given by  
$$
  \psi_I(y_{_{12}}) = x_2x_7, \quad \psi_I(y_{_1}) = x_1x_5, \quad \psi_I(y_{_{23}}) = x_3,\quad \psi_I(y_{_3}) = x_4x_6,
$$  
and $\psi_I(y_A) = 1$ otherwise.
\end{example}

The ring homomorphism $\psi_I$ maps generators of $\Erq$ to generators of $I^r$, preserves least common multiples, and as a result it allows one to study free resolutions of $I^r$ using those of $\Erq$ (\cite{Lr,extremal}). We list some of the relevant properties of $\psi_I$ below. From here on, we use the notation
$$\Nrq=\{\ba=(a_1,\ldots,a_q) \in \NN^q \st a_1 + \cdots + a_q = r\} \qfor q,r \geq 1,
$$
so that $I^r$ is generated by 
$$
\bma=m_1^{a_1}\cdots m_q^{a_q} \qfor \ba=(a_1,\ldots,a_q) \in \Nrq.
$$
In particular $\Erq$ is generated by $\pmea=\e_1^{a_1}\cdots\e_q^{a_q}$ with $\ba=(a_1, \dots, a_q)\in \Nrq$.

 \begin{proposition}
 \label[proposition]{l:Lr-extremal}
 Let $q,r \geq 1$ and let $I$ be an ideal in the polynomial ring $R$, minimally generated by
   square-free monomials $m_1,\ldots,m_q$. Then the following statements hold.
   \begin{enumerate}
   \item $\Erq$ is minimally generated by  $\pmea$ for $\ba \in \Nrq$~\cite[Proposition 7.3(ii)]{Lr}; 
   \item $\psi_I(\pmea)=\bma$ for each $\ba\in \Nrq$~\cite[Lemma 7.7(i)]{Lr};
   \item\label{i:psi-intersection3}
   $\psi_I(\Erq)R  =I^r$~\cite[Lemma 7.7(ii)]{Lr};
   \item If $A\neq B$ are  subsets of $[q]$, then $\psi_I(y_A)$ and $\psi_I(y_B)$ have disjoint supports, that is,
   $$
   \gcd(\psi_I(y_A),\psi_I(y_B))=1 \qforall A\neq B \subseteq [q];
   $$
   \item  For any monomials $\mu,\nu$ in $S_{[q]}$, we have
    $$
    \psi_I(\lcm(\mu,\nu)) = \lcm(\psi_I(\mu),\psi_I(\nu)).
    $$
    In particular, $\psi_I$ provides a map between the lcm lattices of  $\Erq$ and $I^r$ (\cite[Lemma 7.7(iii)]{Lr}) via 
$$
   \psi_I(\lcm(\pme^{\ba_1}, \dots, \pme^{\ba_t}))=\lcm(\bm^{\ba_1}, \dots, \bm^{\ba_t}) \qforall \ba_1,\ldots,\ba_t \in \Nrq, \quad t\ge 1;
$$  
\item  If $\J$ and $\K$ are monomial ideals  in $S_{[q]}$, then 
    $$
    \psi_I(\mathcal{J}\cap \mathcal{K})R = \psi_I(\mathcal{J})R \cap \psi_I(\mathcal{K})R.
    $$

\end{enumerate}
 \end{proposition}
 
\begin{proof}
    Statements $(1),(2),(3)$ are shown in \cite{Lr}.     
    For Statement $(4)$, let $A$, $B$ be subsets of $[q]$ and assume $k\in \unsure_I(A)\cap \unsure_I(B)$. By \cref{l:gamma} we have
    \[
    j\in A\iff x_k\mid m_j\iff j\in B
    \]
    and hence $A=B$. Thus, if $A\ne B$, we conclude that $\unsure_I(A)\cap \unsure_I(B)=\emptyset$, implying $\psi_I(A)$ and $\psi_I(B)$ have disjoint support. 

    For Statement $(5)$, set $\mu=\prod_{\tiny A\subseteq [q]} y_A^{a_A}$ and $\nu=\prod_{\tiny A\subseteq [q]} y_A^{b_A}$, using the convention that $y_\emptyset=1$. We have then 
$$\begin{array}{ll}
\psi_I(\lcm(\mu,\nu)) 
&= \displaystyle \psi_I\Big (\prod_{\tiny A\subseteq [q]} y_A^{\max\{a_A,b_A \}}\Big ) 
= \displaystyle \prod_{\tiny A\subseteq [q]}\psi_I( y_A)^{\max\{a_A,b_A\}}
=\displaystyle \prod_{\tiny A\subseteq [q]}\prod_{\tiny x_k\in \unsure_I(A)}x_k^{\max\{a_A,b_A\}}\\
&\\
&= \displaystyle \lcm\Big (\prod_{\tiny A\subseteq [q]}\psi_I( y_A)^{a_A},\prod_{\tiny A\subseteq [q]}\psi_I( y_A)^{b_A}\Big)
=\lcm(\psi_I(\mu),\psi_I(\nu)).\\
\end{array}
$$
where the first equality follows from the expressions of $\mu$ and $\nu$, the second equality follows from the fact that $\psi_I$ is a ring homomorphism, the third equality follows from the definition of $\psi_I$, the fourth equality follows from the definition of $\psi_I$ and the fact each variable of $R$ appears with a non-zero exponent in at most one term of each products  $\prod_{\tiny A\subseteq [q]}\psi_I( y_A)^{a_A}$ and $\prod_{\tiny A\subseteq [q]}\psi_I( y_A)^{b_A}$ by Statement (4), and the last equality follows from the expressions of $\mu$ and $\nu$ and the fact that $\psi_I$ is a ring homomorphism.

For Statement (6), observe that in order to prove an inclusion between two ideals, it suffices to show that the generators of one ideal are contained in the second ideal. The inclusion  
$$\psi_I(\mathcal{J}\cap \mathcal{K})R \subseteq \psi_I(\mathcal{J})R \cap \psi_I(\mathcal{K})R$$ 
thus follows from the inclusion $\psi_I(\mathcal{J}\cap \mathcal{K})\subseteq \psi_I(\mathcal{J})\cap \psi_I(\mathcal{K})$.

For the reverse inclusion, observe that the monomial generators of $\psi_I(\mathcal{J})R$ and $\psi_I(\mathcal{K})R$ have the form $\psi_I(\nu)$ and  $\psi_I(\mu)$ respectively, where $\nu,\mu$ are monomial generators of $\mathcal{J},\mathcal{K}$ respectively. Then the monomial 
generators of the ideal $\psi_I(\mathcal{J})R  \cap  \psi_I(\mathcal{K})R$ have the form  $\lcm(\psi_I(\nu), \psi_I(\mu))$, and the conclusion follows by applying Statement (5), as indicated below: 
\begin{gather*}
\lcm(\psi_I(\nu), \psi_I(\mu))=\psi_I(\lcm(\nu,\mu))\in \psi_I(\mathcal{J}\cap \mathcal{K}).\qedhere
\end{gather*}
\end{proof}

\section{\bf  Integral Closures via $\psi_I$} \label{sec:integral-closure}

 In this section we show that in order to determine whether the $r$-th power of a square-free monomial ideal $I$ generated by $q$ elements is integrally closed, one only has to examine the image of the generators of the integral closure of ${\E_q}^r$ (\cref{thm:integral-closure-extremal,thm:integral-defect}).  \cref{thm:integral-defect} also shows that the integral defect of $I^r$ is bounded by that of $\Erq$. Since our result applies to every power, extremal ideals turn out to be a powerful tool for studying Rees algebras and normality of square-free monomial ideals. As an application of our results in this section, we show that every square-free monomial ideal generated by at most $3$ elements is normal (\cref{t:general-I-ic}(2)). We provide a theoretical proof of this result, but note that our work allows us to reduce the problem about infinitely many monomial ideals to a setting where a practical computation using the Normaliz~\cite{NormalizSource} package in Macaulay2~\cite{M2} for a single power of a specific ideal could be used to prove the general result.

We start by introducing the necessary background on integral closures of monomial ideals. We refer to \cite{Huneke-Swanson2006} for more details.
Given an ideal $I \subset R = \sfk[x_1, \dots, x_n]$, an element $f$ is said to be \emph{integral over $I$} if for some integer $s$ there exist $a_i \in I^i$ for $1\le i\le s$ such that
$$
    f^s + f^{s-1}a_{1} + \dots + f a_{s - 1} + a_{s} = 0.
$$
The set $\overline I$ of all elements that are integral over $I$ is called the integral closure of $I$. It is well known that $\overline I$ is always an ideal. If $I = \overline I$, we say $I$ is \emph{integrally closed}, and if $I^r$ is integrally closed for all $r\ge 1$, we say $I$ is \emph{normal}.

For monomial ideals, it is possible to translate the problem of computing integral closure to a problem about convex geometry. The \emph{convex hull} of a set $\Delta \subseteq \RR^n$ is the set 
{\begin{equation}\label{d:conv}
    \conv(\Delta) = \{c_1 \ba_1 + \dots + c_s \ba_s \st \sum c_i = 1, \quad \ba_i \in \Delta \qand c_i \geq 0\}.
\end{equation}}
Using the notation  
$$\bx^{\ba}=x_1^{a_1}\cdots x_n^{a_n} \qwhere 
\ba = (a_1, \dots, a_n) \in \NN^n,
$$
we have the following characterization of the integral closure of monomial ideals.

\begin{theorem}[{\cite[Section 1.4]{Huneke-Swanson2006}}]\label[theorem]{l:integral-closure-basics}
Let $J$ be an ideal in a polynomial ring $R=\sfk[x_1, \dots, x_n]$ minimally generated by monomials $\bx^{\ba_1}, \dots, \bx^{\ba_s}$ 
with $s \geq 1$ and $\ba_j=(a_{j1},\ldots,a_{jn})\in \mathbb N^n$ for each $j=1,\ldots,s$.
Let  $\bb=(b_1, \ldots, b_n)\in \NN^n$. Then the following are equivalent: 
\begin{enumerate}
\item $\bx^\bb \in \overline J$;
\item $\bb \in \conv(\ba \in \NN^n \st \bx^{\ba} \in J)$;
\item There exists an integer $i$ such that $(\bx^\bb)^i \in J^i$; that is, there exist monomials $u_1, \dots, u_i$ in $J$ such that 
\[
(\bx^\bb)^i=u_1u_2\cdots u_i;
\]
\item\label{i:int-4} There exist rational non-negative numbers $\gamma_1, \dots, \gamma_s$ such that $\sum_{j=1}^s \gamma_j =1$ and
\[
(b_1, \dots, b_n)\succcurlyeq \sum_{j=1}^s \gamma_j(a_{j1}, \dots, a_{jn})
\]
where $(v_1,\ldots , v_n) \succcurlyeq (w_1, \ldots , w_n)$ if and only if $v_i \ge w_i$ for all $i$.
\end{enumerate}
\end{theorem}

We prove in \cref{thm:integral-closure-extremal} that for any square-free monomial ideal $I$ and any $r\ge 1$, the map $\psi_I$ sends the generators of $\overline{\Erq}$ to the generators of $\overline{I^r}$.  We start by using the above characterizations of integral closures to show how integral closures of monomial ideals behave when deleting a variable.

\begin{lemma}\label[lemma]{l:int-closure-1}
For each positive integer $n$, consider the ring homomorphism 
$$\varphi_n\colon R=\sfk[z_1, \ldots, z_n] \to R'=\sfk[z_1, \ldots, z_{n-1}]
\qwhere
\varphi_n(z_i)=\begin{cases} 1&\text{if\, $i=n$}\\
z_i&\text{if\, $1\le i<n.$}\end{cases} 
$$
 If $J$ is a monomial ideal in $R$, then $\varphi_n(\overline{J})R'
 =\overline{\varphi_n(J) R'}$. 
\end{lemma}
\begin{proof} If $z_n^t$ is a minimal generator of $J$ for some $t \ge 1$, then $$\varphi_n(\overline{J})R'= R'=\phi_n(J)R'
 =\overline{\varphi_n(J) R'}$$ and the result holds.
  Thus we may assume that $\bz^{\ba_1}, \dots, \bz^{\ba_s}$ are the minimal monomial generators of $J$, where $\ba_i=(a_{i1},\dots,a_{in}) \in \mathbb N^n$ for all $i\in [n]$ and $s\ge 1$, and $\supp(\bz^{\ba_i}) \neq \{z_n\}$ for all $i \in[s]$. Note that the ideal $\varphi_n(J)$ is generated by the monomials $\varphi_n(\bz^{\ba_i})=\bz^{\bc_i}$ where 
  \[
  \bc_i=(a_{i1}, \dots, a_{i\, (n-1)})\in \mathbb N^{n-1} 
  \]
for all $i\in [n]$. In particular, $\varphi_n(J)$ is minimally generated by a subset of $\{\bz^{\bc_1}, \dots, \bz^{\bc_s}\}$.

 To prove the inclusion $\subseteq$, let  $f=z_1^{b_1}\cdots z_n^{b_n}\in \overline J$. We will show $\varphi_n(f)\in \overline{\varphi_n(J)}$. By \cref{l:integral-closure-basics} there exists an integer $i$ such that $f^i=u_1\cdots u_i$ for monomials $u_1, \dots, u_i$ in $J$. Since $\varphi_n$ is a ring homomorphism and $f^i\in J^i$, it follows that 
\[
(\varphi_n(f))^i=\varphi_n(f^i)=\varphi_n(u_1\cdots u_i)=\varphi_n(u_1)\cdots \varphi_n(u_i)\in \varphi_n(J)^i.
\]
It then follows from \cref{l:integral-closure-basics} that  $\varphi_n(f)\in \overline{\varphi_n(J)}$.

For the reverse inclusion, let $g=z_1^{b_1}\cdots z_{n-1}^{b_{n-1}}\in \overline{\varphi_n(J)}$. By \cref{l:integral-closure-basics}\eqref{i:int-4}, using the convention that $\gamma_j=0$ if 
$\bz^{\bc_j} = z_1^{a_{j1}}\cdots z_{n-1}^{a_{j(n-1)}}$ is not a minimal generator of $\varphi_n(J)$, there exist non-negative rational numbers $\gamma_1, \dots, \gamma_{s}$ such that
\[
\sum_{j=1}^{s}\gamma_j=1\qand (b_1, \dots, b_{n-1})\succcurlyeq \sum_{j=1}^s \gamma_j(a_{j1}, \dots, a_{j\, (n-1)}).
\] 
Set
\[
b_n=\sum_{j=1}^s a_{jn} \qand f=z_1^{b_1}\cdots z_{n}^{b_{n}}.
\]
We have then $\varphi_n(f)=g$ and 
\[
(b_1, \ldots, b_n)\succcurlyeq \sum_{j=1}^s \gamma_j(a_{j1}, \dots, a_{jn}).
\]
Using \cref{l:integral-closure-basics}, we see that $f\in \overline J$ and we conclude $g\in \varphi_n(\overline J)$. This finishes the proof. 
\end{proof}

Using \cref{l:int-closure-1} we now show how the ring homomorphism $\psi_I$ can be used to compute integral closures of powers of $I$.

\begin{theorem}\label{thm:integral-closure-extremal}
    Let $I$ be an ideal in the
  polynomial ring $R=\sfk[x_1,\ldots,x_n]$ minimally generated by $q\ge 1$ square-free monomials. Then for all $r\ge 1$,  
    \[
    \psi_I(\overline{\Erq})R = \overline{I^r}.
    \]
\end{theorem}
\begin{proof}
 Let $\SqI=\sfk[y_A\colon \psi_I(y_A)\ne 1]$ be a subring of $S_{[q]}$. We write $\psi_I = \psi_1 \circ \psi_2$ where
         \begin{align*}
             \psi_2\colon S_{[q]}\to \SqI  &\quad {\mbox{\rm is given by}} \quad \psi_2(y_A)=\begin{cases} 
                 1 & \text{if } \psi_I(y_A)=1\\
                 y_A & \text{otherwise;}
             \end{cases}\\
             \psi_1\colon \SqI \to R &\quad {\mbox{\rm is given by}} \quad   \psi_1(y_A) = \psi_I(y_A).
         \end{align*}
Let $\mathcal J=\psi_2(\Erq)$ and note that $\mathcal J$ is an ideal in $\SqI$ since $\psi_2$ is surjective. 
A repeated application of \cref{l:int-closure-1} gives 
\begin{equation}
\label{R1}
\psi_2(\overline{\Erq})=\overline{\mathcal J}.
\end{equation}
By \cref{l:Lr-extremal}(4), $\psi_1$ takes each of the variables of $S_{[q],I}$ into a product of distinct variables
in $R$, and different variables of $S_{[q]}$ are mapped to products with disjoint support. Hence \cite[Proposition~2]{IrenaYassemi} 
gives \begin{equation}
\label{R2}
\psi_1(\overline{\mathcal J})R=\overline{\psi_1(\mathcal J)R}.
\end{equation}
Finally, \cref{l:Lr-extremal}\eqref{i:psi-intersection3} gives
\begin{equation}
\label{R3}
\psi_1(\mathcal J)R=\psi(\Erq)R=I^r\,.
\end{equation}
The desired conclusion follows from combining \cref{R1,R2,R3}. 
 \end{proof}

The following corollary shows how $\psi_I$ can be used to check the equality, or lack thereof, of ordinary powers and their integral closures of a given monomial ideal. Since the statement follows immediately from \cref{thm:integral-closure-extremal} using basic facts from \cref{l:Lr-extremal}, we do not include a proof. 

\begin{corollary}\label[corollary]{c:integralclosurepsi}
Let $I$ be an ideal in the ring $R$ minimally generated by $q\ge 1$ square-free monomials and let $r\ge 1$. Then 
$\overline{I^r}=I^r$ if and only if        
    $\psi_I(m) \in I^r$ for all  $m\in \mingens(\overline{\Erq}) \setminus \mingens(\Erq)$. In particular, 
$$\overline{I^r}=I^r+        
    \psi_I(\overline{\Erq}\smallsetminus  \Erq)R.$$
\end{corollary}
 
In general, it is well-known that $I^r \subseteq \overline{I^r}$. As a way to measure how the integral closure of a power differs from the ordinary power, we define an invariant that is the integral closure counterpart to the better-known symbolic defect introduced in~\cite{GGSVT2019}. A similar notion of integral symbolic defect was introduced by Oltsik in~\cite{O2024}.

\begin{definition}\label[definition]{d:idefect}
Let $I$ be a homogeneous ideal in a polynomial ring $R$ and let $r$ be a positive integer. The \emph{$r$-th integral defect} of $I$ is 
$$
    \idefect(r, I) = \mu\left(\overline{I^r}/I^r\right),
$$
where $\mu$ denotes the minimal number of generators as an $R$-module.
\end{definition}
In this paper, we will only consider the case where $I$ is a monomial ideal.
We now show how we can use extremal ideals and the map $\psi_I$ to determine the additional generators of $\overline{I^r}$ and provide a uniform bound for the integral defect.

\begin{theorem}\label{thm:integral-defect}
    Let $I$ be an ideal in the ring $R$ minimally generated by $q$ square-free monomials, and let $r\ge 1$. Then 
    $$\idefect(r, I) \leq \idefect(r, \E_q).$$
\end{theorem}

\begin{proof}
  Set $d=\idefect(r,I)$. Then there exist monomials 
    $$M_1,\dots, M_d\in \overline{I^r}\setminus I^r$$ 
    such that their homomorphic images $\overline{M_1},\dots, \overline{M_d}$ minimally generate $\overline{I^r}/I^r$. Equivalently, $$M_1,\dots, M_d \in \mingens(\overline{I^r})\setminus I^r.$$ Recall from \cref{thm:integral-closure-extremal} that $\psi_I(\overline{\Erq})R=\overline{I^r}$. In particular, this implies that $\{\psi_I(E)\colon E\in \mingens(\overline{\Erq}) \}$ generates $\overline{I^r}$. As $\psi_I(\overline{\Erq})$ is a set of monomials, we can choose a subset that is a minimal monomial generating set of $\overline{I^r}$. In other words, there exist monomials $E_1,\dots, E_d\in \mingens(\overline{\Erq})$ such that $\psi_I(E_i)=M_i$ for any $i\in [d]$. Moreover, recall from \cref{l:Lr-extremal}(3) that $\psi_I(\Erq)R=I^r$. Since $M_i\notin I^r$ for any $i\in [d]$, we must have $E_i\notin \Erq$ for any $i\in [d]$. Therefore, by definition, we have $d\leq \idefect(r,\E_q)$, as~desired. 
\end{proof}

\cref{thm:integral-closure-extremal,thm:integral-defect} show that the study of integral closures of powers of $\E_q$ can be an extremely useful tool for understanding integral closures of powers of arbitrary square-free monomial ideals. Thus we will explore means of identifying elements of the integral closure of $\Erq$  in small cases.  First, for $q=3$, we work towards showing that $\E_3$ is normal (\cref{t:E32}).

We start by translating the criterion (1)$\iff$(4) in \cref{l:integral-closure-basics}. 
 We define  basis vectors for $\NN^{2^q-1}$ by $\be_A$ for $\emptyset \neq A \subseteq [q]$ with 
 \begin{equation}\label{eq:eA}
     \be_A: 2^{[q]}\ssm \{\emptyset\} \to \{0,1\} \qwhere \be_A(B)=\begin{cases}1 \qif B=A\\0 \qif B \neq A.
\end{cases}
 \end{equation}
 This allows us to write any monomial $\bm \in S_{[q]}$ as
\begin{equation}\label{eq:additivenotation} 
\bm = \prod_{\emptyset \ne A\subseteq [q]} y_A^{b_A} = \by^{\bb}, \qwhere \bb 
=\sum_{\emptyset \ne A \subseteq [q]} b_A \be_A.
\end{equation} Additionally, for $\ba=(a_1,\dots, a_q)\in \Nrq$ we define for each $\emptyset \ne A \subseteq [q]$
\begin{equation}
\label{eq:|a|}
|\ba|_A=\sum_{i\in A}a_i.
\end{equation}
 Then
\begin{equation}
\label{eq:ea}
\pmea=\e_1^{a_1}\cdots\e_q^{a_q}=\prod_{\emptyset\ne A\subseteq [q]}y_A^{|\ba|_A}
=\by^{\widetilde{\ba}} \qwhere \widetilde{\ba}=\sum_{\emptyset \ne A \subseteq [q]} |\ba|_A\be_A\,.
\end{equation}

\begin{example}
Let $q=3$. Consider $\ba = (2,1,3)$. Then 
$$\pmea = \e_1^2\e_2\e_3^3 = (y_{_1}^2y_{_{12}}^2y_{_{13}}^2y_{_{123}}^2)(y_{_2}y_{_{12}}y_{_{23}}y_{_{123}})(y_{_3}^3y_{_{13}}^3y_{_{23}}^3y_{_{123}}^3) = y_{_1}^2y_{_2}y_{_3}^3y_{_{12}}^3y_{_{13}}^5y_{_{23}}^4y_{_{123}}^6.$$
Note that $\tilde{\ba}$ corresponds to the exponent vector of the final expression. In general, there is not a fixed order imposed on the $y_A$, so it is useful to view $\tilde{\ba}$ using the summation notation:
$$
\tilde{\ba}=2\be_1+\be_2+3\be_3+3\be_{12}+5\be_{13}+4\be_{23}+6\be_{123}.
$$

\end{example}

Using the notation above, we now prove the aforementioned translation, providing a concrete means of viewing membership in $\overline{\Erq}$ in terms of a system of linear equations and inequalities.

\begin{lemma}\label[lemma]{l:closure-r-alpha}
   Let $\bb=\sum_{\emptyset \ne A\subseteq [q]} b_A\be_A\in \mathbb N^{2^q-1}$  and $r\ge 1$. 
    Then $\by^\bb \in \overline{{\E_q}^r}$  if and only if there exist rational non-negative numbers $\alpha_1,\ldots,\alpha_q$ such that 
    \begin{enumerate}
    \item\label{i:closure-1} $\alpha_1 + \cdots + \alpha_q = r$ and 
    \item\label{i:closure-2} $b_A \geq \sum_{i \in A} \alpha_i \qforall \emptyset \ne A \subseteq [q].$
    \end{enumerate}
\end{lemma}
\begin{proof}
The minimal generators of $\Erq$ have the form  
$$ \pmea=\e_1^{a_1}\cdots\e_q^{a_q} \qwhere
\ba=(a_1,\dots, a_q)\in \Nrq.$$  
By \cref{l:integral-closure-basics} and using \cref{eq:ea}, we have $\by^\bb \in \overline{{\E_q}^r}$ for $\bb=\sum_{\emptyset\ne A\subseteq [q]}b_A\be_A$ if and only if there exist non-negative rational numbers $\gamma_{\ba}$ with $\ba\in \Nrq$ such that 
\begin{equation}\label{eq:bA-r-gamma}
\sum_{\ba\in \Nrq}\gamma_{\ba} = 1
\qand  
            b_A \geq \sum_{\ba\in \Nrq} \gamma_{\ba}|\ba|_A\qforall \emptyset \ne A \subseteq [q].
    \end{equation}
    
Before explicitly dealing with the two implications in our statement, we record some computations that will help connect \cref{eq:bA-r-gamma} to the inequality in~\eqref{i:closure-2}. Assume there exist non-negative rational numbers $\gamma_{\ba}$ for $\ba \in \Nrq$, and $\alpha_i$ for $i \in [q]$, that satisfy
$$(\alpha_1, \ldots, \alpha_q) = \sum_{\ba\in \Nrq} \gamma_\ba \ba.$$
Viewing this vector equation component-wise yields a system of $q$ equalities
    \begin{equation}\label{eq:bi-r-gamma}
        \alpha_i = \sum_{\ba=(a_1, \dots, a_q) \in \Nrq} \gamma_\ba a_i \qforall i\in [q].
    \end{equation}

    Using \cref{eq:bi-r-gamma}, we compute the right hand side of the inequality in \cref{eq:bA-r-gamma} for a set $\emptyset \ne A \subseteq [q]$ as follows: 
\begin{equation}\label{eq:coef}
\sum_{\ba\in \Nrq} \gamma_{\ba}|\ba|_A=
\sum_{\substack{\ba\in \Nrq\\\ba =  (a_1, \ldots, a_q)}}\sum_{i\in A}\gamma_{\ba}a_i=
\sum_{i\in A}\sum_{\substack{\ba\in \Nrq\\\ba =  (a_1, \ldots, a_q)}}\gamma_{\ba}a_i=\sum_{i\in A}\alpha_i.
\end{equation}
In particular, when $A = [q]$, then $|\ba|_{[q]}=r$ for all $\ba=(a_1, \dots, a_q)\in \Nrq$, so \cref{eq:coef} (in reverse order) simplifies to 
\[
    \sum_{i \in [q]} \alpha_i = r \sum_{\ba\in \Nrq} \gamma_\ba.
\]
Consequently, given $(\gamma_{\ba})_{\ba\in \Nrq}$ and $(\alpha_i)_{i \in [q]}$ satisfying \cref{eq:bi-r-gamma}, we have: 
\begin{align}
\label{eq:gamma-alpha1}
    \sum_{\ba \in \Nrq}\gamma_{\ba} = 1 & \quad\iff \quad\sum_{i\in [q]}\alpha_i=r\\
    \label{eq:gamma-alpha2}
      b_A \geq \sum_{\ba \in \Nrq} \gamma_{\ba}|\ba|_A &\quad \iff\quad  b_A\ge \sum_{i\in A}\alpha_i.
    \end{align}
    
Now assume $\by^\bb \in \overline{{\E_q}^r}$, and hence there exist non-negative rational numbers $\gamma_{\ba}$ satisfying \cref{eq:bA-r-gamma}. Define non-negative rational numbers $\alpha_i$ by \cref{eq:bi-r-gamma}. These integers then satisfy the conditions~\eqref{i:closure-1} and~\eqref{i:closure-2}, in view of \cref{eq:gamma-alpha1,eq:gamma-alpha2}, as desired. 
    
    Conversely, if there exist $\alpha_i$'s satisfying conditions~\eqref{i:closure-1} and~\eqref{i:closure-2}, then we define
    \[
 \gamma_{\ba}=\begin{cases}
        \frac{1}{r}\alpha_j &\text{if $\ba=r\be_j$}\\
        0 &\text{otherwise}
    \end{cases}
    \]
    for $\ba \in \Nrq$ where $\be_i$ is the $i$th standard basis vector in ${\mathbb R}^q$. Then for each $i\in [q]$:
    \[
        \sum_{\ba=(a_1, \ldots,a_q) \in \Nrq} \gamma_{\ba} a_i = \sum_{j \in [q]} \gamma_{r\be_j} (r\be_j)_i = r\gamma_{r\be_i} = \alpha_i,
    \]
    so \cref{eq:bi-r-gamma} holds. Then 
    \cref{eq:gamma-alpha1,eq:gamma-alpha2} show that \cref{eq:bA-r-gamma} holds. Thus $\by^\bb \in \overline{{\E_q}^r}$.
\end{proof}

Consider a generator $\by^\bb$ of $\overline{\Erq}$, where $\bb=\sum_{\emptyset \ne A\subseteq [q]} b_A\be_A$. We next show there is a relationship among the exponents of variables $y_A$, that is, the integers $b_A$, when the sets indexing the variables satisfy a containment relation.

\begin{lemma}\label[lemma]{l:bXY}
    If $\by^\bb$ is a minimal generator of $\overline{{\E_q}^r}$, then $\emptyset\ne S \subseteq T \subset [q]$ implies $b_S \leq b_T$.
\end{lemma}
\begin{proof}
    Assume $\by^\bb$ is a minimal generator of $\overline{{\E_q}^r}$, and that $S \subseteq T$.  Assume by way of contradiction that $b_S >b_T$.  Since $b_T$ and $b_S$ are integers, $b_S - 1 \geq b_T$. Since  $\by^\bb \in \overline{{\E_q}^r}$, by \cref{l:closure-r-alpha}, there exist non-negative rational numbers $\alpha_1, \ldots, \alpha_q$ satisfying the two conditions of \cref{l:closure-r-alpha}.
    Then
    \[
        b_S - 1 \geq b_T \geq \sum_{i \in T} \alpha_i \geq \sum_{i \in S} \alpha_i,
    \]
     where the third inequality follows from $S \subseteq T$ and each $\alpha_i$ being non-negative.

    As a result, the vector $\bb' = \bb - \be_S$ also satisfies the conditions of \cref{l:closure-r-alpha} with $\alpha_1,\ldots,\alpha_q$, and so $\by^{\bb'} \in \overline{{\E_q}^r}$.  But $\by^{\bb'}$ divides $\by^{\bb}$, which contradicts $\by^\bb$ being a minimal generator of $\overline{{\E_q}^r}$, as desired.
\end{proof}

In order to prove ${\E_3}^r = \overline{{\E_3}^r}$ in \cref{t:E32} below, we will use the \emph{analytic spread}, $\ell(I)$, of a monomial ideal $I$. For a general definition of analytic spread, see~\cite[Definition 5.1.5]{Huneke-Swanson2006}. In the monomial case when the generators are of the same degree, the analytic spread can be computed as the rank of a matrix built from the monomial generators, as described in the proof below.

\begin{lemma}
\label[lemma]{l:analytic spread}
The analytic spread of $\E_q$ is given by $\ell(\E_q) = q$. 
\end{lemma}
\begin{proof}
It is well known (see~\cite[Proposition 7.1.7, Exercise 7.4.10]{villarrealbook}) that when $I$ is a monomial ideal generated by elements of the same degree, then the analytic spread $\ell(I)$ can be computed as the rank of a matrix where the columns are indexed by the variables of $R$, and the rows are indexed by generators of $I$, and the $ij$ entry of this matrix is $1$ if the generator $i$ is divisible by variable $j$, and $0$ otherwise. For the extremal ideal $\E_q$, this matrix has rank $q$ since the minor formed by selecting the columns indexed by $y_i$ for $1 \le i \le q$ forms the $q \times q$ identity matrix. Since every $\e_i$ has the same degree, we conclude $\ell(\E_q) = q$ for every $q$.   
\end{proof}

\begin{remark}
\label[remark]{Singla}
Let $I$ be a monomial ideal with analytic spread $\ell$. Singla~\cite[Corollary 5.3]{S2007} showed that if the ideals  $I, I^2, \dots, I^{\ell - 1}$ are integrally closed, then $I$ is normal, that is, $I^r$ is integrally closed for all $r\ge 1$. 
In particular, any ideal $I$ generated by at most two square-free monomials is normal.  Indeed, we have $\ell(I)\le 2$, and thus Singla's result shows that it suffices to argue that $I$ is integrally closed. This is true, since any square-free monomial ideal is radical, and hence integrally closed by \cite[Remark~1.1.3(4)]{Huneke-Swanson2006}. 
\end{remark}

By \cref{Singla}, in order to show that $\E_3$ is normal we only need to check that ${\E_3}^2$ is integrally closed. This could be shown by a Macaulay2~\cite{M2} computation, but we include a proof in \cref{t:E32} in order to show how the technique of focusing on exponents can be used. This provides a relatively small template of how the system of inequalities will be used in later proofs.

\begin{proposition}
\label[proposition]{t:E32}
   For every $r\ge 1$ we have ${\E_3}^r = \overline{{\E_3}^r}$. In other words, $\E_3$ is normal.  
\end{proposition} 
\begin{proof}
    By \cref{Singla}, we only need to check that ${\E_3}^2$ is integrally closed.
Let $\by^\bb$ be a minimal generator of $\overline{{\E_3}^2}$ , where
    \[
\by^\bb=y_{_1}^{b_1}y_{_2}^{b_2}y_{_3}^{b_3}y_{_{12}}^{b_{{12}}}y_{_{13}}^{b_{13}}y_{_{23}}^{b_{{23}}}y_{_{123}}^{b_{{123}}}. 
    \]
We show that $\by^\bb \in {\E_3}^2$. 
By \cref{l:closure-r-alpha} there exist non-negative rational numbers $\alpha_1$, $\alpha_2$, $\alpha_3$  such that the following inequalities hold for all distinct $i,j$ with $1\le i, j \le 3$ 
\begin{align*}
b_{123}&\ge \alpha_1+\alpha_2+\alpha_3=2\\
b_{ij}&\ge \alpha_i+\alpha_j \\
b_i&\ge \alpha_i. 
\end{align*}
Assume first $b_i=0$ for some $i$. Without loss of generality, say $b_1=0$. This implies $\alpha_1=0$ and $\alpha_2+\alpha_3=2$, which means that $b_{23}\ge 2$. If $b_2\ge 1$ and $b_3\ge 1$, then $b_{12}\ge 1$ and $b_{13}\ge 1$ by \cref{l:bXY}. It follows that $\e_2\e_3\mid \by^\bb$, implying $\by^\bb \in {\E_3}^2$. Hence we may assume $b_2=0$ or $b_3=0$. Without loss of generality, say $b_2=0$. Then $\alpha_2=0$ and $\alpha_3=2$. This implies $b_3, b_{13}, b_{23}, b_{123} \ge 2$. In this case $\e_3^2\mid \by^\bb$, hence  $\by^\bb \in {\E_3}^2$.

Assume now $b_i\ge 1$ for all $i$. By \cref{l:bXY}, we have then $b_{ij}\ge 1$ for all $i\ne j$.  If $b_{12} = b_{13} = b_{23} = 1$, then we have
\[
3=b_{12}+b_{13}+b_{23}\ge (\alpha_1+\alpha_2)+(\alpha_1+\alpha_3)+(\alpha_2+\alpha_3)=4,
\]
a contradiction.  We conclude that $b_{ij} \geq 2$ for some $i\ne j$. Without loss of generality, assume $b_{12} \geq 2$. Then  $b_{1}, b_{2}, b_{13}, b_{23} \geq 1$ and $b_{12}, b_{123}  \geq 2$ together imply that $\by^\bb$ is a multiple of $\e_1\e_2 \in {\E_3}^2$, and so $\by^\bb \in {\E_3}^2$. Thus ${\E_3}^2$ is integrally closed, as desired, implying that $\E_3$ is normal by~\cref{Singla}.
\end{proof}

We will show that $\E_q$ is not normal when $q\ge 4$. This is accomplished in \cref{p:not normal}, after first noting some useful facts.

\begin{remark}
\label[remark]{rem:in-or-out}
In order to check whether $\by^\bb\in \Erq$, we need to find (or prove it does not exist) $\ba\in \Nrq$ such that $\pmea\mid\by^\bb$. Using the notation in \cref{eq:|a|},  we have
\begin{enumerate}
\item $\by^\bb\in \Erq$ if and only if there exists $\ba\in \Nrq$ such that 
\[
b_A\ge |\ba|_A\qforall \emptyset\ne A\subseteq [q].
\]
\item  $\by^\bb\notin \Erq$ if and only if for every $\ba\in \Nrq$ there exists $\emptyset\ne A\subseteq [q]$ such that 
\[
b_A< |\ba|_A.
\]
\end{enumerate}
\end{remark}

For $r,q \in \mathbb{N}$ with $r \ge 2$ and $q \ge 4$, define 

\begin{equation}\label{e:newfrq}
    g(r,q)\coloneqq 
    \Big( \prod_{\scriptsize \substack{A\subseteq [q]\\ |A|\geq 1}} y_A \Big )
    \Big( \prod_{\scriptsize \substack{A\subseteq [q]\\ |A|\geq 3}} y_A  \Big ) 
    \e_1^{r-2}.
\end{equation}

\begin{proposition}
\label[proposition]{p:not normal}
If $q\ge 4$ and $r\geq 2$, then $g(r,q)$ satisfies the integral dependence equation $$z^q-(g(r,q))^q=0$$ over $\Erq$. In particular,
$g(r,q) \in \overline{{\E_q}^r}\smallsetminus {\E_q}^r$.
\end{proposition}
\begin{proof}
    We first consider the case $r=2$. To show that $g(2,q)\in \overline{{\E_q}^2}$, it suffices to show $(g(2,q))^q\in {\E_q}^{2q}$ by \cref{l:integral-closure-basics}(3). We use \cref{rem:in-or-out}(1) and choose $\ba=(2,2,\dots, 2)\in \N_q^{2q}$. Set $\by^\bb = g(2,q)$ so, for a set $\emptyset\neq A\subseteq [q]$,  \[
    b_A=\begin{cases}
        1 & \text{if } |A|\leq 2,\\
        2&\text{otherwise.}
    \end{cases}\]
    If $|A|\leq 2$, we have
\[
 (q\bb)_A = qb_A =q \geq 4 \geq 2|A|=|\ba|_A.
\]
    On the other hand, when $|A|\geq 3$, we have
    \[
    (q\bb)_A = qb_A =2q  \geq 2|A|=|\ba|_A.
    \]
    Thus $(g(2,q))^q\in {\E_q}^{2q}$, which implies that $g(2,q)\in \overline{{\E_q}^2}$, as claimed.
    
    We now have $g(r,q)=g(2,q) \e_1^{r-2}\in \overline{\Erq}$ by \cref{l:integral-closure-basics}(3). 
    Next we show that $g(r,q)\notin \Erq$ using \cref{rem:in-or-out}~(2). Abusing notation, we set $\by^\bb=g(r,q)$. We note that for $\emptyset \neq A\subseteq [q]$, we have
    \[
    b_A=\begin{cases}
        1 & \text{if } |A|\leq 2 \text{ and } 1\notin A,\\
        2 & \text{if } |A|\geq 3 \text{ and } 1\notin A,\\
        r-1 & \text{if } |A|\leq 2 \text{ and } 1\in A,\\
        r & \text{if } |A|\geq 3 \text{ and } 1\in A.
    \end{cases}
    \]
    
    Let $\ba\in \Nrq$. If there exist $2 \le i < j \le q$ such that $a_i,a_j >0$, take $A=\{i,j\}$. Then
    \[
    b_{A}=1 < 2\leq a_i+a_j=|\ba|_A.
    \]
    If there exists a unique $i\in [q]$ such that $a_i> 0$, take $A=\{i\}$. If $i\neq 1$, then
    \[
    b_{A}=1 < r= a_i=|\ba|_A.
    \]
    On the other hand, if $i=1$, then 
    \[
    b_{A}=r-1 < r= a_i=|\ba|_A.
    \]
    We can now assume that there exists an integer $i\geq 2$ such that $a_k>0$ if and only if $k\in \{1,i\}$. Take $A=\{1,i\}$. Then
    \[
    b_{A}=r-1 < r=|\ba|_A.
    \]
Thus $g(r,q) \in \overline{{\E_q}^r}\smallsetminus {\E_q}^r$.
\end{proof}

In \cref{e:M2-comps} below we refer to computations of the integral closure of ${\E_q}^r$ that are done using Macaulay2~\cite{M2}. Observe that  the integral closure of a monomial ideal does not depend on the base field, and thus one can run these computations over any field. Such computations are fast when $q$ and $r$ are small, but they become impractical as these parameters increase.

\begin{example}
\label[example]{e:M2-comps}
Using the definition of 
$g(r,q)$  in \cref{e:newfrq},
\begin{equation}\label{eq:f24}
g(2,4)= y_{_1}y_{_2}y_{_3}y_{_4}y_{_{12}} y_{_{13}} y_{_{14}}y_{_{23}} y_{_{24}}y_{_{34}} y_{_{123}}^2 y_{_{124}}^2y_{_{134}}^2y_{_{234}}^2y_{_{1234}}^2.
\end{equation}
Using Macaulay2~\cite{M2} computations, we observe: 
\begin{enumerate}
\item  $\overline{{\E_4}^2}={\E_4}^2 + (g(2,4))$;
\item if $q\in \{4,5\}$, then $\overline{{\E_q}^{q-1}}= \overline{{\E_q}^{q-2}}\E_q$. 
\end{enumerate}
\end{example}

Part (1) of this example in particular shows that $g(2,4)$ is the only monomial in $\mingens(\overline{{\E_4}^2})\setminus \mingens({\E_4}^2)$, and thus \cref{p:not normal} is optimal when $q=4$ and $r=2$. On the other hand, experiments show that even with $q=5$, $\mingens(\overline{{\E_5}^2})\setminus \mingens({\E_5}^2)$ has more than one~monomial.

\cref{t:general-I-ic}  below translates our results about  $\E_q$ into statements that hold for any $I$ generated by $q$ square-free monomials. The map $\psi_I$ is instrumental in this translation. Our statements give concrete criteria for computing $\overline{I^r}$, checking whether $I^r$ is integrally closed and whether $I$ is~normal.

\begin{theorem}\label{t:general-I-ic}
Let $I$ be an  ideal in the polynomial ring $R$ with $q \ge 1$ minimal square-free monomial generators $m_1, \dots, m_q$, and let  $g(r,q)$ be as defined in \cref{e:newfrq}.
The following hold:
\begin{enumerate}
\item\label{i:p-geq-4} if $q \ge 4$, $r\ge 2$ and $\psi_I(g(r,q))\notin I^r$, then $\overline{I^r}\ne I^r$;
\item\label{i:q=3}  if $q\le3$, then $I$ is a normal ideal;
\item\label{i:g(2,4)} if $q=4$, then $\overline{I^2}= I^2+(\psi_I(g(2,4)))$;
\item\label{i:q-4-5}  if $q \in \{4,5\}$, then $I$ is normal if and only if $\overline{I^{r}}=I^r$ for all $2\le r\le q-2$;
\item\label{i:q-4-normal-equiv} if $q=4$, then the following are equivalent:
\begin{enumerate}
\item $I$ is a normal ideal;
\item $\psi_I(g(2,4)) \in I^2$;
\item $\psi_I(y_{ij})=1$ for some $1\leq i<j\leq 4$;
\item $\gcd(m_a,m_b)\mid \lcm(m_c,m_d)$ for some $\{a,b,c,d\}=\{1,2,3,4\}$.  
\end{enumerate}
\end{enumerate}
\end{theorem}

\begin{proof}
Since $g(r,q)\in \overline{\Erq}$ by \cref{p:not normal}, then $\psi_I(g(r,q))\in \overline{I^r}$ by \cref{thm:integral-closure-extremal}. Thus Statement~\eqref{i:p-geq-4} holds. 

By \cref{Singla}, to show Statement~\eqref{i:q=3} it suffices to assume $q=3$. \cref{t:E32} gives that $\E_3$ is normal, that is, $\overline{{\E_3}^r}={\E_3}^r$ for all $r\ge 2$. Then \cref{thm:integral-closure-extremal} and \cref{l:Lr-extremal}(3) imply  $\overline{I^r}=I^r$ for all $r\ge 2$, establishing Statement \eqref{i:q=3}. 

Assuming $q=4$, \cref{e:M2-comps}(1) shows that $\overline{{\E_4}^2}={\E_4}^2+ (g(2,4))$. Statement~\eqref{i:g(2,4)} follows by applying the ring homomorphism $\psi_I$ to this equality and using \cref{thm:integral-closure-extremal} and \cref{l:Lr-extremal}(3). 
  
  To prove Statement \eqref{i:q-4-5}, note that if $I$ is normal then $\overline{I^r}=I^r$ for all $r$, and so in particular the equality holds for the given range of $r$. For the other implication, assume  $q\in \{4,5\}$ and $\overline{I^{r}}=I^r$ for all $2\le r\le q-2$. As noted in \cref{Singla}, a square-free monomial ideal $I$ with $q$ generators is normal if and only if $\overline{I^r}=I^r$ for all $r\leq \ell(I)-1$. Since $\ell(I)\le q$ and $I$ is integrally closed, we thus have: 
\begin{equation}
\label{eq:r for normal}
\text{$I$ is normal}\iff \overline{I^{r}}=I^r \qforall 2\le r\le q-1.
\end{equation}
To show that $I$ is normal, it suffices thus to show that $\overline{I^{q-1}}=I^{q-1}$. We show this using a series of equalities below, where the equalities follow in order from: \cref{thm:integral-closure-extremal};  \cref{e:M2-comps}(2) since $q \in \{4,5\}$; the fact that $\psi_I$ is a ring map; \cref{l:Lr-extremal}(3) and
     \cref{thm:integral-closure-extremal}; the assumption when $r=q-2$; and finally from basic definitions 
\[
\overline{I^{q-1}}=\psi_I(\overline{{\E_q}^{q-1}})R=\psi_I(\overline{{\E_q}^{q-2}}\E_q)R=\!\left(\psi_I(\overline{{\E_q}^{q-2}})R\right)\!\left(\psi_I(\E_q)R\right)\!=\overline{I^{q-2}}I=(I^{q-2})I=I^{q-1}.
\]


To show Statement \eqref{i:q-4-normal-equiv}, assume $q=4$. We first observe (c)$\iff$(d) is a direct consequence of the definition of the map $\psi_I$. The equivalence of (a) and (b) is a consequence of Statements \eqref{i:g(2,4)} and \eqref{i:q-4-5}.  Thus, it suffices to show (b)$\iff$(c). 

Assume (b), i.e.~$\psi_I(g(2,4))\in I^2$, and hence $\psi_I(g(2,4))$ is divisible by a minimal monomial generator of $I^2$. Assume $m_im_k\mid \psi_I(g(2,4))$, where $i,k\in [4]$. Choose $j\in [4]\ssm \{i\}$ such that  if $k\ne i$ then $j=k$. Observe that $y_{ij}^2\mid \e_i\e_k$ and hence $\psi_I(y_{ij})^2\mid \psi_I(\e_i\e_j)$. Since $m_im_j=\psi_I(\e_i\e_j)$, we conclude that $\psi(y_{ij}^2)\mid \psi_I(g(2,4))$ and hence $\psi_I(y_{ij})\mid \psi_I(h)$ where 
\[
h= \frac{g(2,4)}{y_{ij}}.
\]
Assume $\psi_I(y_{ij})\ne 1$ and let $x$ be a variable in the support of $\psi_I(y_{ij})$. Then we must have $x\mid \psi_I(y_A)$ for some $y_A$ is the support of $h$. Observing that $h$ is not divisible by $y_{ij}$, we must have $A\ne \{i,j\}$. However $\psi_I(y_{ij})$ and $\psi_I(y_A)$ have disjoint supports, see \cref{l:Lr-extremal}(4). We must have thus $\psi_I(y_{ij})=1$, so (c) holds.

Conversely, assume $\psi_I(y_{ij})=1$ and $i\ne j$.  Using the formula for $g(2,4)$ in \cref{eq:f24}, observe that 
$\psi_I(\e_i\e_j)\mid \psi_I(g(2,4))$. Therefore  $\psi_I(g(2,4))\in\psi_I({\E_q}^2)R=I^2$. 
\end{proof}

In view of \cref{t:general-I-ic}\eqref{i:q-4-5} above, we ask the following: 
\begin{question}\label[question]{q:integralclosure}
With $I$ as in \cref{t:general-I-ic}, is it true for all $q\ge 4$ 
that \cref{eq:r for normal} can be improved to
\[
\text{$I$ is normal}\iff \overline{I^{r}}=I^r \qforall 2\le r\le q-2\ ?
\]
\end{question}
If $I$ is a square-free monomial ideal minimally generated by $q$ elements with analytic spread $\ell < q$, a positive answer to~\cref{q:integralclosure} follows directly from Singla's result~\cite{Sin08}. The interesting case is when the analytic spread of $I$ is exactly $q$, in which case Singla's result states that one need only check the equality up to power $q-1$ of $I$, but a positive answer to~\cref{q:integralclosure} would show it suffices to check only up to power $q - 2$. Note that~\cref{t:general-I-ic}\eqref{i:q-4-5} gives a positive answer to the question when $q = 4,5$.

\section{\bf Symbolic Powers and $\psi_I$ }\label[section]{sec:symbolic}

In this section we show that to understand primary decompositions of square-free monomial ideals, one only has to study primary decompositions of the extremal ideals $\E_q$. We first study the behavior of primary decompositions under $\psi_I$  (\cref{t:primary-to-primary}), and then provide applications of the result to the study of symbolic powers of square-free monomial ideals (\cref{thm:symbolic-extremal}), symbolic~defects (\cref{c:sdefectbound}), and containment problems including resurgence and asymptotic resurgence (\cref{c:rho}).

As above, $R=\sfk[x_1,\ldots,x_n]$ is a polynomial ring over a field $\sfk$. We recall some basic facts and definitions regarding primary decompositions, and refer the reader to \cite{villarrealbook} for additional details.
An ideal $Q$ of  $R$ is {\bf primary} if $ab \in Q$ and $a \not\in Q$ implies that  $b^r \in Q$ for some power $r$. The ideal  $Q$ is {\bf irreducible} if there are no ideals $I$ and $J$ such that $Q = I \cap J$, unless $I$ or $J$ equals $Q$.
When $Q$ is a monomial ideal, and therefore has a unique minimal monomial generating set, $Q$ is primary if and only if every variable in the support of any minimal generator appears as a pure power in the generating set. In addition, (see, e.g.,~\cite[Proposition~5.1.16]{villarrealbook}) 
 $Q$ is irreducible if and only if 
$$
Q = (x_{i_1}^{a_{i_1}},\ldots,x_{i_s}^{a_{i_s}})
$$
for some integer $s >0$. In particular, an irreducible monomial ideal is primary. Every ideal $I$ of $R$ can be written as a (not necessarily unique)  intersection of primary ideals $Q_i$ 
\begin{equation}\label{eq:pd}
    I = Q_1 \cap \cdots \cap Q_t
\end{equation}
which is called a {\bf primary decomposition} of $I$.
The primary decomposition is {\bf irredundant} if for each $j \in [t]$,  $$\bigcap_{i \in [t] \setminus \{j\}} Q_i \not \subseteq Q_j.$$ When  \cref{eq:pd} is an irreducible primary decomposition of a monomial ideal, then it is straightforward to check that the decomposition is irredundant if and only if 
$$Q_i \not\subseteq Q_j \qfor i \ne j,$$
see also \cite[Definition 3.3.4]{SSW}.

While a primary decomposition of an ideal need not be unique, every monomial ideal has a {\bf unique  irredundant irreducible} primary decomposition, that is, an irredundant decomposition  as in~\cref{eq:pd}
where each $Q_i$ is an irreducible monomial ideal (\cite[Theorem~5.1.17]{villarrealbook}).

If  \cref{eq:pd} is a primary decomposition of $I$, then for each $i\in [t]$ the prime ideal $\sqrt{Q_i}$ is called an {\bf associated prime} of $I$. The set of associated primes of $I$ is denoted by $\ass(R/I)$ or  $\ass(I)$. It is well-known that $\ass(I)$ does not depend on a chosen primary decomposition of $I$ (see, e.g., \cite[page 10]{HH11}). The {\bf minimal} primes of $I$ are the associated primes of $I$ which are minimal with respect to inclusion. We denote this set by $\minimalprimes(I)$.

We now examine how $\psi_I$ behaves on irreducible primary monomial ideals. Below and throughout the paper, for an ideal $J$ of $S_{[q]}$, we often use the notation $\psi_I(J)$  to refer to the {\it ideal generated by} $\psi_I(J)$ in the polynomial ring $R$, or more precisely, $\psi_I(J)R$. In the case where $\psi_I$ is surjective, we automatically have that $\psi_I(J)$ and $\psi_I(J)R$ are equal as sets.

\begin{lemma}\label[lemma]{prop:primary-decomposition-powers-variables}
    Let $I$ be an ideal in the polynomial ring  $R$ minimally generated by $q\ge 1$ square-free monomials.  Let 
    \[
    Q=(y_{A_1}^{r_1},\dots, y_{A_h}^{r_h})
    \]
be an irreducible primary monomial ideal in $S_{[q]}$ for some distinct nonempty sets $A_1,\dots, A_h \subseteq [q]$ and $r_i>0$. If  $\unsure_I(A_i)= \emptyset$ for some $i\in [h]$, then $\psi_I(Q)=R$. Otherwise,
\begin{equation}
 \label{e: psi-Q}
\psi_I(Q)=  \bigcap_{\substack{(k_1, \ldots, k_h)\\x_{k_{i}}\in\unsure_I(A_i), \, i \in [h]}} ( x_{k_1} ^{r_1},\dots, x_{k_h} ^{r_h}  )
\end{equation}  
is an irredundant irreducible primary decomposition of $\psi_I(Q)$.
\end{lemma}

\begin{proof}
  If $\unsure_I(A_i)$ is empty for some $i\in [h]$, then $\psi_I(y_{A_i})=1$ and hence $\psi_I(Q)=(1)=R$, and we are done. So assume that $\unsure_I(A_i)$ is non-empty for all $i\in [h]$. 
Observe that $\psi_I(Q)$ is equal to
 $$  \left(\psi_I(y_{A_1})^{r_1},\dots, \psi_I(y_{A_h})^{r_h}\right)
        =  \Big(\!\!\! \prod_{\tiny x_{k}\in \unsure_I(A_1) } x_{k} ^{r_1},\dots, \!\!\prod_{\tiny x_{k} \in \unsure_I(A_h)} x_{k} ^{r_h} \! \Big )
        = \!\!\!\!\!\! \bigcap_{\substack{(k_1, \ldots, k_h)\\x_{k_{i}}\in\unsure_I(A_i), \, i\in [h]}}\!\!\!\!\!\! ( x_{k_{1}} ^{r_1},\dots, x_{k_{h}} ^{r_h}  ).
        $$

   The last expression is an irreducible decomposition of $\psi_I(Q)$. To see that the decomposition is irredundant, assume $Q_1=(x_{k_1}^{r_1},\dots, x_{k_h}^{r_h})$ and $Q_2=(x_{s_1}^{r_1},\dots, x_{s_h}^{r_h})$ are components of this decomposition with $Q_1\subseteq Q_2$. It suffices to prove $Q_1=Q_2$. Since $\unsure_I(A_i)$ are pairwise disjoint, these are minimal monomial generating sets of $Q_1$ and $Q_2$. Consider $j\in [h]$. Since $x_{k_j}^{r_j}\in Q_1$,  it is a multiple of $x_{s_t}^{r_t}$ for some $t \in [h]$ and in particular $x_{k_j}=x_{s_t}$. Since $\unsure_I(A_i)$ are pairwise disjoint, $t=j$ and $ s_t=k_j$. Therefore,
   \[
   Q_1=(x_{k_1}^{r_1},\dots, x_{k_h}^{r_h})= (x_{s_1}^{r_1},\dots, x_{s_h}^{r_h})=Q_2,
   \]
   as desired.
\end{proof}

We apply \cref{prop:primary-decomposition-powers-variables} to show how the irredundant irreducible primary decomposition of $I^r$ can be obtained from $\Erq$. 

\begin{theorem}[{\bf Primary decomposition of $I^r$ via $\psi_I$}]\label{t:primary-to-primary}
   Let $I$ be an ideal in the polynomial ring  $R$ minimally generated by $q\ge 1$ square-free monomials and let $r\ge 1$. If 
    $\Erq$ has an irredundant irreducible  primary decomposition
    \[
    \Erq = Q_1\cap \cdots \cap Q_{t} \qwith  
    Q_i = (y_{A_{1,i}}^{r_{1,i}},\dots, y_{A_{s_i,i}}^{r_{s_i,i}}) \qwhere 
    \emptyset\ne A_{j,i} \subseteq [q]
    \]
   then the irredundant irreducible  primary decomposition of $I^r$ is 
    \[
 I^r= \bigcap_{i\in L} \bigcap_{\substack{(k_{1,i}, \ldots, k_{s_i ,i})\\x_{k_{j,i}}\in\unsure_I(A_{j,i}), \,  j \in [s_i]}}\left( x_{k_{1,i}} ^{r_{1,i}},\dots, x_{k_{s_i,i}} ^{r_{s_i,i}}  \right)
 \qwhere
 L=\{i \in [t] \st \psi_I(Q_i) \neq R\}.
    \]
\end{theorem}

\begin{proof}
 Using \cref{l:Lr-extremal} and \cref{e: psi-Q}, we have
    \[
    I^r=\psi_I(\Erq) = \bigcap_{i\in L} \psi_I(Q_i)=\bigcap_{i\in L} \bigcap_{\substack{(k_{1,i}, \ldots, k_{s_i ,i})\\x_{k_{j,i}}\in\unsure_I(A_{j,i}) \, \forall j \in [s_i]}}\left( x_{k_{1,i}} ^{r_{1,i}},\dots, x_{k_{s_i,i}} ^{r_{s_i,i}}  \right)
    \]
   and the expression on the right is an irreducible primary decomposition. We now show there are no containments among the ideals in this decomposition.     
    Suppose 
    \begin{align*}
        ( x_{k_{1,u}} ^{r_{1,u}},\dots, x_{k_{s_u,u}} ^{r_{s_u,u}}  ) \subsetneq ( x_{k_{1,v}} ^{r_{1,v}},\dots, x_{k_{s_v,v}} ^{r_{s_v,v}}  )
    \end{align*}
    for some  $u,v\in L$. By relabeling, we can assume that $s_u\leq s_v$, $k_{j,u}=k_{j,v}$, and $r_{j,u}\leq r_{j,v}$ for any $j\in [s_u]$. By \cref{prop:primary-decomposition-powers-variables}, we must have $u\neq v$. Since $k_{j,u}=k_{j,v}$ for $j \in [u]$ and the sets $\unsure_I(A_{j,i})$ are pairwise disjoint, it follows  that $A_{j,u}=A_{j,v}$ for $j\in [u]$. Thus
    $Q_u\subseteq Q_v$, a contradiction since the decomposition of $\Erq$ was assumed to be irredundant.
\end{proof}

For additional information regarding how $\psi_I$ interacts with primary decompositions, associated primes, persistence properties, and related algebraic questions, as well as computations of associated primes of $\Erq$, see \cite{AlgII}. Knowing that $\psi_I$ sends a primary decomposition of $\Erq$ to one of $I^r$ strongly suggests that $\psi_I$ will be useful in the study of symbolic powers. 

\begin{definition}[{\protect\cite[Definition 4.3.22 and Proposition 4.3.25]{villarrealbook}}]\label[definition]{d:symbolic-powers}
    Let $I$ be an ideal in a polynomial ring $R=\sfk[x_1,\dots, x_n]$ where $\sfk$ is a field, and let $r$ be a positive integer. Assume that 
    \[
    I^r=\fq_1\cap \cdots \cap \fq_t
    \]
    is a primary decomposition of $I^r$. Order the primary components so that $\sqrt{\fq_i}$ is a minimal prime of $I$ if and only if $i\in [t']$ for some $t' \leq t$. Then the \emph{$r$-th symbolic power} of $I$, denoted by $I^{(r)}$, is defined to be
    \[
    I^{(r)}= \fq_1\cap \cdots \cap \fq_{t'}.
    \]
When $I$ is a square-free monomial ideal and $\fp_1,\ldots,\fp_t$ are the minimal (or equivalently, associated) primes of $I$, then we also have 
$$
    I^{(r)}= \fp_1^r\cap \cdots \cap \fp_t^r.   
    $$ 
\end{definition}

Using this definition, we are able to show that the map $\psi_I$ behaves nicely with regard to symbolic~powers.

\begin{theorem}[{\bf $\psi_I$ preserves symbolic powers}]\label{thm:symbolic-extremal}
        Let $I$ be an ideal in the
  polynomial ring $R=\sfk[x_1,\ldots,x_n]$ minimally generated by $q\ge 1$ square-free monomials. Then for all $r\ge 1$, 
    $$\psi_I({\E_q}^{(r)}) = I^{(r)}.$$ 
\end{theorem}

\begin{proof}
Let $\Erq=Q_1\cap \cdots \cap Q_t$ be an irredundant irreducible  primary decomposition. By definition, we have
    \[
    {\E_q}^{(r)}
    = 
    \bigcap_{\sqrt{Q_i}\in \minimalprimes(\Erq)} Q_i
   .
    \]
    Applying the map $\psi_I$ to both sides of this equation, we obtain
    \begin{equation}\label{eq:symb-1}
        \psi_I({\E_q}^{(r)}) =  \bigcap_{\sqrt{Q_i}\in \minimalprimes(\Erq)} \psi_I(Q_i) = \bigcap_{i\in L, \sqrt{Q_i}\in \minimalprimes(\Erq)} \psi_I(Q_i)
    \end{equation}
   where $L$ is the set of indices for which $\psi_I(Q_i) \ne R$. Recall that in \cref{t:primary-to-primary} we have
    \[
    I^r = \bigcap_{i\in L} \psi_I(Q_i).
    \]
    Applying \cref{t:primary-to-primary} with $r=1$ shows that minimal primes of $I^r$ arise from those of $\Erq$, so we have
    \begin{equation}\label{eq:symb-2}
        I^{(r)} = \bigcap_{i\in L, \sqrt{Q_i}\in\minimalprimes(\Erq)} \psi_I(Q_i).
    \end{equation}
    The result now follows from \cref{eq:symb-1,eq:symb-2}.
\end{proof}
    
An immediate consequence of~\cref{thm:symbolic-extremal} is that extremal ideals play a key role in the theory of symbolic powers of square-free monomial ideals. The following corollary shows how $\psi_I$ can be used to check the equality, or lack thereof, of symbolic powers and ordinary powers of a given ideal. Since the statement follows immediately from \cref{thm:symbolic-extremal} using basic facts from \cref{l:Lr-extremal}, we do not include a~proof. 

\begin{corollary}\label[corollary]{c:symbolic=ordinary}
Let $I$ be an ideal in the ring $R$ minimally generated by $q\ge 1$ square-free monomials, and let $r\ge 1$. Then $I^{(r)}=I^r$ if and only if $\psi_I(m)\in I^r$ for all  $m\in \mingens({\E_q}^{(r)}) \setminus \mingens(\Erq)$. In~particular, 
$$I^{(r)}=I^r+        
    \psi_I({\E_q}^{(r)}\smallsetminus  \Erq)R.$$
\end{corollary}

When $I$ is generated by $2$ monomials, the situation is particularly nice.

\begin{corollary}\label[corollary]{c:E2-symbolic}
    Let $I$ be an ideal generated by $2$ square-free monomials and let $r$ be an integer. Then  $I^{(r)}=I^r$ for any $r\geq 1$.
\end{corollary}
\begin{proof}
    It suffices to prove the result for $I=\E_2$, as the general statement then follows from \cref{thm:symbolic-extremal}. We have
    \begin{align*}
        {\E_2}^r=(y_{_1}y_{_{12}},y_{_2}y_{_{12}})^r = y_{_{12}}^r(y_{_1},y_{_2})^r= (y_{_{12}})^r \cap (y_{_1},y_{_2})^r= \E_2^{(r)},
    \end{align*}
    as desired.
\end{proof}

 A key goal in the theory of symbolic powers is to understand which elements are in a symbolic power but not in an ordinary power. In 2018 Galetto, Geramita, Shun and Van Tuyl~\cite{GGSVT2019} introduced the \emph{$r$-th symbolic defect} of an ideal $I$ 
$$
    \sdefect(r, I) = \mu(I^{(r)}/I^r),
$$
where $\mu$ denotes number of minimal generators, as a way to measure how far the $r$-th symbolic power is from the $r$-th ordinary power of $I$. Extremal ideals provide a means of bounding the symbolic defects of all square-free monomial ideals.

\begin{theorem}\label{c:sdefectbound}
    Let $I$ be a square-free monomial ideal minimally generated by $q$ elements. Then for any $r \ge 1$
    $$
        \sdefect(r, I) \leq \sdefect(r, \E_q).
    $$
\end{theorem}

\begin{proof}
The proof is identical to the proof of
\cref{thm:integral-defect}, up to replacing integral closures of powers with symbolic powers and using $\psi_I({\E_q}^{(r)})=I^{(r)}$ from \cref{thm:symbolic-extremal}. 
\end{proof}

\begin{example}\label[example]{ex:p3}
  For $q=3, r=2$ a direct computation via Macaulay2~\cite{M2} shows that
    \[
    {\E_3}^{(2)} = {\E_3}^2 + (y_{_1}y_{_2}y_{_3}y_{_{12}}y_{_{13}}y_{_{23}}y_{_{123}}^2).
    \]
     Thus by \cref{c:sdefectbound}, $\sdefect(2, I) \leq 1$ for any ideal
     $I$ generated by three square-free monomials.
     By \cref{c:symbolic=ordinary}, $I^{(2)}=I^2$ if and only if  $\psi_I(y_{_1}y_{_2}y_{_3}y_{_{12}}y_{_{13}}y_{_{23}}y_{_{123}}^2) \in \psi_I({\E_3}^2)=I^2$. 
 \end{example} 

 It follows from~\cref{thm:integral-closure-extremal,thm:symbolic-extremal} that extremal ideals play a key role in the study of containments of different notions of powers of (square-free monomial) ideals. 
 An active area of study for symbolic powers is the \emph{containment problem}, which seeks to identify pairs of integers $a, b$ such that $I^{(b)} \subseteq I^a$.  Examples of invariants that were defined in order to study containment problems of an ideal $I$ are the \emph{resurgence} 
 $$
    \rho(I) = \sup \big \{\frac{s}{r}\  \st s,r \in \NN \qand  I^{(s)} \not \subset I^r \big \} 
 $$  
 defined by Bocci and Harbourne~\cite{BH10}, and the \emph{asymptotic resurgence}  
   $$ \rho_a(I) = \sup \big \{\frac{s}{r} \st s,r \in \NN \qand  I^{(st)} \not \subset I^{rt} \qforall t \gg 0 \big\}
 $$
 defined by Guardo, Harbourne and Van Tuyl~\cite{GHVT2013}.
 It was then shown in~\cite[Section 4]{DFMS2019} that 
 $$
    \rho_a(I) = \sup \big \{\frac{s}{r} : s, r \in \NN \qand I^{(s)} \not \subset \overline{I^r} \big \}.
 $$
Extremal ideals bound both of these invariants, as seen in the next result.

 \begin{corollary}\label[corollary]{c:rho}
     Let $I$ be a square-free monomial ideal minimally generated by $q$ elements.
     Then 
     $$
        \rho(I) \leq \rho(\E_q) \qand \rho_a(I) \leq \rho_a(\E_q).
     $$
 \end{corollary}

 \begin{proof}
     We first prove the result for $\rho_a$. By~\cref{thm:integral-closure-extremal,thm:symbolic-extremal} we know 
     $$
        I^{(s)} \not \subset \overline{I^r} \implies {\E_q}^{(s)} \not \subset \overline{{\E_q}^r}. 
     $$
     In particular, we conclude 
     \begin{equation*}
        \left\{\frac{s}{r}: s, r \in \NN \qand I^{(s)} \not \subset \overline{I^r}\right\} \subset \left\{\frac{s}{r} : s,r \in \NN \qand {\E_q}^{(s)} \not \subset \overline{{\E_q}^r}\right\}. 
     \end{equation*}
     The proof for resurgence is analogous, but instead of~\cref{thm:integral-closure-extremal} we use~\cref{l:Lr-extremal}\eqref{i:psi-intersection3}.
 \end{proof}

 \begin{example}
     Computations using the Normaliz~\cite{NormalizSource} and the SymbolicPowers~\cite{SymbolicPackage} packages from Macaulay2~\cite{M2} show that 
     $$
         {\E_4}^{(7)} \subset \overline{{\E_4}^5}.
     $$
     This computation shows that given any square-free monomial ideal $I$ generated by $4$ elements, we have $I^{(7)} \subseteq \overline{I^5}$.
     See \cref{fig:enter-label} for additional results of these calculations.
     \begin{figure}
    \centering
    \includegraphics[width=1\linewidth]{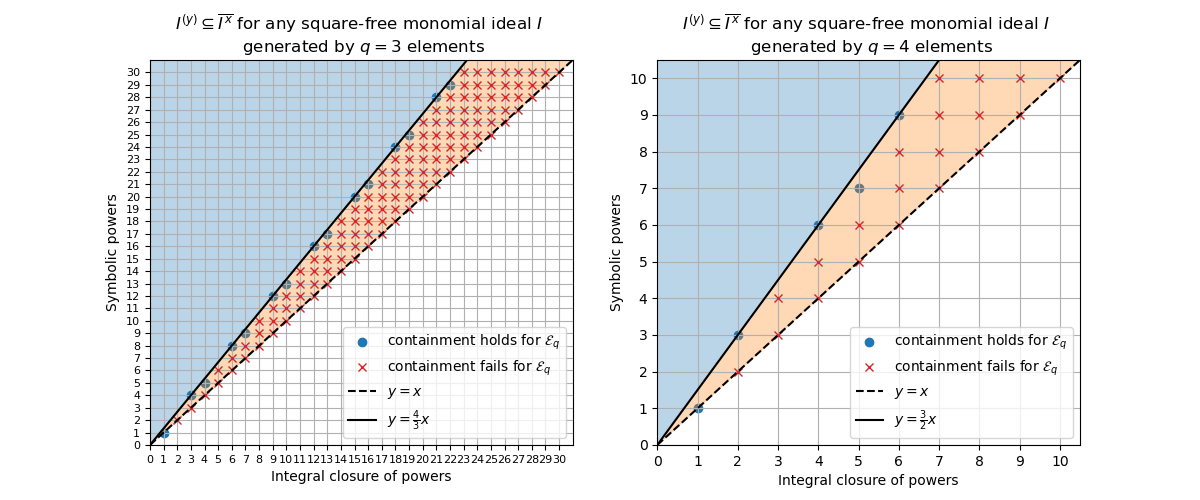}
    \caption{Dots and points above the solid line represent pairs $(x, y)$ such that $I^{(x)} \subseteq \overline{I^y}$ for every square-free monomial ideal generated by $q$ elements, and an $\times$ or point below the dotted line means ${\E_q}^{(x)} \not\subseteq \overline{{\E_q}^y}$.}
    \label{fig:enter-label}
\end{figure}
 \end{example}

    In 2023 Villarreal~\cite{V2023} studied the asymptotic resurgence of square-free monomial ideals from a geometric perspective, and in doing so, provided a concrete method for computing asymptotic resurgence of square-free monomial ideals as the minimum value of certain dot products of vectors (see~\cite[Theorem 3.7]{V2023}). 
Following~\cite[Procedure A.1]{V2023}, and then applying~\cite[Theorem 3.7]{V2023} allows us to compute $\rho_a(\E_q)$ for low values of $q$ where a direct computation is feasible. When $q=3$, using that $\E_3$ is normal by~\cref{t:general-I-ic}\eqref{i:q=3}, we are also able to compute $\rho(\E_3)$. Combining these computations with~\cref{c:rho}, we conclude the following.

\begin{proposition}\label[proposition]{p:resurgence-cases}
 Let $I$ be a square-free monomial ideal minimally generated by $q$ elements.
Then
\begin{eqnarray*}
\rho(I)=\rho_a(I) \leq \rho(\E_3)=\rho_a(\E_3) = 4/3 &\qwhen q=3;\\
\rho_a(I) \leq \rho_a(\E_4) = 3/2 &\qwhen q=4;\\
\rho_a(I) \leq \rho_a(\E_5) =8/5  &\qwhen q=5.
   \end{eqnarray*}
\end{proposition}

Extremal ideals are a viable theoretical tool for studying uniform containments for arbitrary square-free monomial ideals generated by $q$ elements. By \cref{thm:integral-closure-extremal,thm:symbolic-extremal}, if ${\E_q}^{(r)} \subseteq \overline{\Erq}$ for some $r,q$ then $I^{(r)}\subseteq \overline{I^r}$ for all square-free monomial ideals with $q$ generators. However, if ${\E_q}^{(r)} \not\subseteq \overline{\Erq}$, then containment may or may not hold for an ideal generated by $q$ square-free monomials.

\begin{question}[{\bf Containments of types of powers of square-free monomial ideals}]\label[question]{q:containment}
    Given a fixed $q \ge 1$, for which $r, s$ is it true that $I^{(s)} \subseteq \overline{I^r}$ for {\it every} square-free monomial ideal generated by $q$ elements? In other words, what is the asymptotic resurgence of $\E_q$?
\end{question}

\section{\bf  The Second Symbolic Defect}\label[section]{sec:2nd-symbolic-powers}

This section focuses on the second symbolic power of any square-free monomial ideal $I$ via the map $\psi_I$. In \cref{thm:gens-2nd-symbolic}, we provide a concrete formula for computing the generators of $I^{(2)}$. In~\cref{c:explicit-bound} we give an explicit computation of the second symbolic defect of $\E_q$, which leads to a sharp upper bound for the second symbolic defect of arbitrary square-free monomial ideals.  Recall that in our setting of square-free monomial ideals, $\Ass(I)=\minimalprimes(I)$, allowing us to use both algebraic and geometric tools  to study symbolic powers.

We begin by introducing our main tool, the symbolic polyhedron defined in~\cite{CEHH2017}. 
Given a convex set $A$ and an integer $r \ge 1$, the \emph{dilation} of $A$ by $r$ is the convex set 
$$
    rA = \{\ba_1 + \dots + \ba_r \st \ba_i \in A \qforall i\}.
$$
Note that if $A$ can be written as the intersection of convex sets $A_1, \dots, A_s$, then $rA = \bigcap_{i=1}^s r A_i$.
For a square-free monomial ideal $I$, let 
$$
    \Ll(I) = \{\ba \in \NN^n \st \bx^{\ba} \in I\}.
$$
The \emph{symbolic polyhedron} of $I$ is the convex set 
$$
    \Pp(I) = \bigcap_{P \in \Ass(I)}\conv(\Ll(P)).
$$
Note that since multiplication of monomials corresponds to addition of exponents, taking powers of ideals corresponds to dilations of the convex sets defined by the monomials, as in \cref{d:conv}. In particular, $r\conv(\Ll(I)) = \conv(\Ll(I^r))$ for any monomial ideal $I$. Thus the symbolic polyhedron of $I$ allows us to translate questions about containment of elements in the symbolic powers of $I$ to questions about lattice points of dilations of $\Pp(I)$. 
The precise statement appears in~\cref{l:symbolicpolyhedron}.
One direction of~\cref{l:symbolicpolyhedron} is shown in greater generality in~\cite{CEHH2017} for arbitrary monomial ideals. Since we also need the converse for square-free monomial ideals, and this part of the statement is implicit in~\cite{CEHH2017}, we include a proof below.

\begin{lemma}[{\cite[Theorem 5.4]{CEHH2017}}]\label[lemma]{l:symbolicpolyhedron}
    Let $I \subset \sfk[x_1, \dots, x_n]$ be a square-free monomial ideal, $r \ge 1$ an integer, and $\ba \in \NN^n$. Then $\bx^{\ba} \in I^{(r)}$ if and only if $\ba \in r\Pp(I)$.
\end{lemma}

\begin{proof}
    We only need to show that if $\ba \in r \Pp(I)$, then $\bx^{\ba} \in I^{(r)}$, as the other direction is a special case of~\cite[Theorem 5.4]{CEHH2017}.
    If $\ba \in r \Pp(I)$, then $\ba \in r\conv(\Ll(P)) = \conv(\Ll(P^r))$ for every $P \in \Ass(I) = \minimalprimes(I)$. In particular, we have that $\bx^{\ba} \in P^r$ for every minimal prime $P$ of $I$, and hence $\bx^{\ba} \in I^{(r)}$.
\end{proof}

A direct corollary of \cref{l:symbolicpolyhedron} is the following result that allows us to translate the problem of computing the symbolic powers of $\E_q$ to a linear algebra problem. In the context of extremal ideals, the symbolic polyhedron of $\E_q$ is a convex set in $\RR^{2^q - 1}$. When dealing with this polyhedron, we will
use
the notation in \cref{eq:additivenotation}.

Recall that a collection $\C$ of nonempty subsets of $[q]$ is a {\bf minimal set cover} of $[q]$ if 
$$[q] = \bigcup_{A \in \C}A, \qand 
A \not\subseteq \bigcup_{A' \ne A} A' \qforall A \in \C.
$$ 
We use $\CC_q$ to denote the set of all minimal set covers of $[q]$. We will also use the notation $\compl{A}$ to denote the {\bf complement} $[q]\ssm A$ of a set $A\subseteq [q]$.  

\begin{lemma}\label[lemma]{l:symboliclinalg}
    Let $q \ge 1$ and $\bb=\displaystyle \sum_{ \emptyset \neq A\subseteq [q] } b_A\be_A\in \mathbb N^{2^q-1}$.  
    \begin{enumerate}    
    \item\label{i:symbolic-equiv} The following are equivalent:
        \begin{enumerate}
        \item[(a)] $\by^{\bb} \in {\E_q}^{(2)}$;
        \item[(b)] $\bb \in 2\Pp(\E_q)$;
        \item[(c)]\label{i:3} $\sum_{A \in \C} b_{A} \geq 2$ for every $\C \in \CC_q$.
        \end{enumerate}
     \item\label{i:symbolic-min} If $\by^{\bb}$ is a minimal generator of ${\E_q}^{(2)}$, then
     \begin{enumerate}
     \item[(a)] $b_{A} \leq 2$ for all $A$;
     \item[(b)] $b_{[q]} = 2.$
     \end{enumerate}
      \item\label{i:bAT} If $\by^{\bb} \in {\E_q}^{(2)}$ and $b_A = 0$ for some $A$,  then $b_T \geq 2$ whenever $\compl{A} \subseteq T\subseteq [q]$.    
      \end{enumerate}
     \end{lemma}

\begin{proof}
    To see the equivalence in \eqref{i:symbolic-equiv}, note that 
    \cref{l:symbolicpolyhedron}  gives $(a) \iff (b)$. Observe that for a given $\C \in \CC_q$, the inequality in $(c)$ is equivalent to the existence of an $A\in \C$ such that $b_A \geq 2$ or the existence of distinct $A,A'\in \C$ such that $b_A,b_{A'} \geq 1$. In other words, $(c)$ is equivalent to $\by^{\bb} \in (y_A \st A \in \C)^2$ for every $\C \in \CC_q$.  
    On the other hand, it is straightforward to see that $P$ is a minimal prime of $\E_q$ if and only if $P = (y_A \st A \in \C)$ for some $\C$ in $\CC_q$ (cf.\ also \cite{AlgII}). Hence, by the definition of symbolic powers, we conclude $(a) \iff (c)$.

    To see \eqref{i:symbolic-min}(a), assume $b_A > 2$ for some $A$. Since $\by^{\bb} \in {\E_q}^{(2)}$, by \eqref{i:symbolic-equiv}, we have $\frac{\by^{\bb}}{y_{A}} \in {\E_q}^{(2)}$, and in particular $\by^{\bb}$ is not a minimal generator of ${\E_q}^{(2)}$. To see \eqref{i:symbolic-min}(b), notice that $\{[q]\} \in \CC_q$, so since $\by^\bb \in {\E_q}^{(2)}$, by \eqref{i:symbolic-equiv}, $b_{[q]} \ge 2$ and the result follows from \eqref{i:symbolic-min}(a).

    For  \eqref{i:bAT}, assume $\by^{\bb} \in {\E_q}^{(2)}$ and $b_A=0$.  Since $\{[q]\} \in \CC_q$, by \eqref{i:symbolic-equiv} we have $b_{[q]} \ge 2$. Thus  $b_A=0$ implies $A \ne [q]$ and $T \ne \emptyset$. Similarly, if $T=[q]$, then $b_T \ge 2$ and the inequality holds. For the remaining cases, $\{A, T\}$ forms a minimal set cover of 
    $[q]$, and~\eqref{i:symbolic-equiv} implies
    $b_A + b_T \geq 2$, 
    from which the result follows.
\end{proof}

The main result of this section is an explicit description of the minimal generators of ${\E_q}^{(2)}$ for any $q \ge 1$. As a consequence, we obtain a formula for the second symbolic defect of $\E_q$ for any $q \ge 1$. We now introduce a formula, which we will show uniquely describes every minimal generator of ${\E_q}^{(2)}$.   

\begin{definition}
Set $y_\emptyset=1$ and let $\bj = \sum_{\emptyset \neq A \subseteq [q]}\be_A$ be the vector with all entries equal to $1$. For any $q \ge 1$ and each $\emptyset \neq X\subseteq [q]$, we define
\begin{equation}\label{eq:f-defn}
    f_X = \by^{\bj} \prod_{X \subseteq A } \frac{y_{A}}{y_{\compl{A}}}
        = \by^\bj \frac{\displaystyle\prod_{X\subseteq A}y_A}{\displaystyle\prod_{A\subseteq \compl{X}}y_A}=\prod_{X \subseteq A} {y_A}^2 \prod_{\substack{X \cap A \neq \emptyset\\ X \cap \compl{A} \neq \emptyset}} y_A 
\end{equation}
where all the sets $A$ appearing above are assumed to be nonempty subsets of $[q]$.
Define 
\begin{equation}\label{eq:mingens-symbolic2}
\Omega_q^2\coloneqq \{f_X \colon  \emptyset \neq X\subseteq [q]\}.
\end{equation}
\end{definition}

We note some important features of $f_X$ below.
\begin{lemma}\label[lemma]{e:smallX}
Let $q \ge 1$, $i, j \in [q]$ and $\emptyset \ne X,Y\subseteq [q]$.  Then 
\begin{enumerate}
    \item $f_{\{i\}} = {\e_i}^2$;
    \item $f_{\{i,j\}}= \e_i\e_j$;
    \item $f_{[q]}=\by^{\bj+\be_{[q]}}$;
    \item $\deg(f_X)=2^q$;
   \item $f_X\neq f_Y$ if $X\neq Y$;
   \item $|X| \le 2$ if and only if $f_X \in {\E_q}^2$.
\end{enumerate}
\end{lemma}
\begin{proof}
With the indices $A$ and $A_i$ below representing  nonempty  subsets of $[q]$, we first prove Statements~(1)--(4). For~(1), observe that
$$
     f_{\{i\}} = \by^{\bj} {\displaystyle \prod_{\tiny i \in A } \frac{y_{A}}{y_{\compl{A}}}}
         = \frac{\displaystyle \prod_{\tiny A } y_A \prod_{\tiny i \in A }y_A}{\displaystyle \prod_{\tiny i \not\in A }y_A}
         = \big ({\displaystyle \prod_{\tiny i \in A }}y_A\big)^2 = \e_i^2.
$$
For~(2) we can assume that $i \neq j$, as otherwise the statement is proved in~(1). We have
$$\begin{array}{rl}
f_{\{i,j\}} &=\by^{\bj} {\displaystyle \prod_{\tiny \{i,j\} \subseteq A } \frac{y_{A}}{y_{\compl{A}}}}
    = \frac{\displaystyle \prod_{\tiny  A } y_A \prod_{\tiny \{i,j\} \subseteq A }y_A}{\displaystyle \prod_{\tiny \{i,j\}\cap A = \emptyset}y_A}\\
     &\\
    &=  \displaystyle \prod_{\tiny
    \substack{\{i,j\} \cap A_1 =\{i\}\\
              \{i,j\} \cap A_2 = \{j\}\\
              \{i,j\} \cap A_3=\{i,j\}}}
              y_{A_1}  y_{A_2}  {y_{A_3}}^2 
            =  \big( \prod_{ i\in  A }y_A \big) 
              \big( \prod_{ j\in A}y_A \big) =\e_i\e_j.
\end{array}
$$               
Statement~(3) can be seen from 
$$
    f_{[q]} =\by^{\bj} \displaystyle \prod_{\tiny  [q]\subseteq A} \frac{y_{A}}{y_{\compl{A}}}=\by^{\bj}\frac{y_{[q]}}{y_{\emptyset}} =\by^{\bj+\be_{[q]}}.
$$
For Statement (4), let $X$ be a nonempty subset of $[q]$. Note that $\deg(y_{[q]})=1$, $\deg(y_{\emptyset})=0$, and if $A \notin 
\{[q], \emptyset\}$, then $\deg (y_A) = \deg (y_{\compl{A}}) =1$. Thus
$$\begin{array}{rl}
\deg(f_X)
= &\deg(\by^{\bj}) + \displaystyle \Big ( \sum_{X \subseteq A \subsetneq [q]}\big( \deg(y_A) - \deg(y_{\compl{A}}) \big ) \Big ) + (\deg(y_{[q]})-\deg(y_{\emptyset}))\\
= &(2^q - 1)  + \displaystyle \Big ( \sum_{X \subseteq A \subsetneq [q]}\big( 1 - 1 \big ) \Big )+ (1-0)= 2^q,
\end{array}
$$
as desired.

For Statement (5), assume that $X\neq Y$ for nonempty $X,Y\subseteq [q]$. Without loss of generality, assume that $X\not\subseteq Y$. Then $y_Y^2\mid f_Y$ but $y_Y^2 \nmid f_X$. Thus $f_X \ne f_Y$.

For Statement (6), first note that ${\E_q}^2$ is minimally generated by $\e_i\e_j$ with  $1 \le i \leq j \le q$ (see \cref{l:Lr-extremal}). If $|X| \le 2$, it then follows immediately from Statements (1) and (2) that $f_X \in {\E_q}^2$. If $|X| > 2$, then by (4) a degree argument shows $f_X$ is not a multiple of $f_{\{i\}}$ or $f_{\{i,j\}}$ for any $i,j \in [q]$. Since by (5), $f_X$ is also not equal to these elements, $f_X \not\in {\E_q}^2$ and the result follows.
\end{proof}

Before attacking our main result of the section, we prove a technical lemma. 
%
  %

\begin{lemma}\label[lemma]{l:claim}
Assume 
$$\by^\bb \in \mingens({\E_q}^{(2)}) \qwith \bb = \sum_{\emptyset \ne A \subseteq [q]} b_A\be_A \in \mathbb{N}^{2^q-1} .$$ 
      Set 
      $\DD \coloneqq \{ A\subseteq [q]\colon b_A=0 \}.$
  Then  $\DD = \emptyset$ or $\DD$ has a maximum element with respect to set inclusion. 
\end{lemma}

\begin{proof}
It suffices to show if $A_1, A_2 \in \DD$, then $A_1 \cup A_2 \in \DD$.
    Assume to the contrary that there exist $A_1, A_2 \in \DD$ such that $A_1 \cup A_2 \not \in \DD$. Then  $b_{A_1 \cup A_2} > 0$ by the definition of $\DD$. 
    Define $$\by^{\bb'} = \frac{\by^{\bb}}{y_{A_1 \cup A_2}} \qwith \bb' = \bb - \be_{A_1 \cup A_2}.$$
    We will show that $\by^{\bb'} \in {\E_q}^{(2)}$, contradicting the assumption that $\by^{\bb}$ is a minimal generator of ${\E_q}^{(2)}$, by demonstrating that the inequality in \cref{l:symboliclinalg}\eqref{i:symbolic-equiv} holds for every $\C \in \CC_q$.  
    Let $\C \in \CC_q$.  
    
    If $A_1 \cup A_2 \not\in \C$, then 
        $$\sum_{A \in \C} b'_{A} = \sum_{A \in \C } b_{A} \geq 2$$
        by \cref{l:symboliclinalg}\eqref{i:symbolic-equiv} since $\by^{\bb} \in {\E_q}^{(2)}$.
       
    If $A_1\cup A_2\in \C$,  then $\C = \{A_1 \cup A_2\} \cup \K$, for some $\K \subseteq 2^{[q]}$. Since $\C$ is a set cover of $[q]$, 
        $$(A_1 \cup A_2) \cup \bigcup_{K \in \K} K = [q]$$ so $\{A_1, A_2\} \cup \K$ is also set cover of $[q]$, and thus contains a minimal set cover $\C'$. Now $\K$ is not a set cover of $[q]$, so $\C'$ must contain $A_1$ or $A_2$ or both. On the other hand, since $\C$ is minimal,
       $$K'\not\subseteq A_1\cup A_2 \cup \bigcup_{\tiny \substack{K \in \K\\K\neq K'}} K \qforall K' \in \K$$
       which means all elements of $\K$ must remain in $\C'$. It follows that 
        $$  \C'= \quad \{A_1, A_2\} \cup \K \qor   \{A_2\} \cup \K \qor \{A_1\} \cup \K.$$
        Recall that  $\by^{\bb}\in {\E_{q}}^{(2)}$. By a combination of \cref{l:symboliclinalg}\eqref{i:symbolic-equiv} and the fact that $b_{A_1} = b_{A_2} = 0$,  
        $$\sum_{A \in \K} b_{A} = \sum_{A \in \C'} b_A\geq 2.$$
         We then have
        \[
        \sum_{A \in \C} b'_{A}= b'_{A_1\cup A_2}+\sum_{A \in \K} b'_{A} \geq \sum_{A \in \K} b'_{A}= \sum_{A \in \K} b_{A}\geq 2.
        \] 
        Therefore, by \cref{l:symboliclinalg}\eqref{i:symbolic-equiv}, we have 
        $$ \frac{\by^{\bb}}{y_{A_1 \cup A_2}} = \by^{\bb'} \in {\E_q}^{(2)},$$
        contradicting the minimality of $\by^{\bb}$. Thus $A_1 \cup A_2 \in \DD$.
\end{proof}

Next we show that the monomials $f_X$ described above form a minimal generating set of ${\E_q}^{(2)}$. A combination of  \cref{thm:gens-2nd-symbolic,e:smallX} then shows that 
$$\mingens({\E_q}^2) \subseteq \mingens({\E_q}^{(2)})=  \Omega_q^2.$$

It is worth  highlighting that the minimal generators of ${\E_q}^2$ remain minimal generators of ${\E_q}^{(2)}$, which is unusual when dealing with symbolic powers. Often, the symbolic power $I^{(2)}$ has lower degree generators which divide some of the generators of the ordinary power $I^2$. However, the generators of ${\E_q}^{(2)}$ all have the same degree by \cref{e:smallX}, and the degree variance for a general $I$ is addressed by the $\psi_I$ map.

\begin{theorem}\label{thm:gens-2nd-symbolic}
    For any $q \ge 1$, we have
    \[
    \mingens({\E_q}^{(2)}) =  \Omega_q^2. 
    \]
\end{theorem}

\begin{proof}  
We show each inclusion separately.
\subsubsection*{Proof of $\Omega_q^2 \subseteq \mingens({\E_q}^{(2)})$} We first  show that $f_X$ is in ${\E_q}^{(2)}$ for any nonempty subset $X$ of $[q]$. If $|X|=1$ then by \cref{e:smallX}, $f_X\in {\E_q}^2\subseteq {\E_q}^{(2)}$ as desired. Now assume $|X|\geq 2$, in which case, $q \ge 2$ as well. Set  
     $$f_X=\by^{\bb} \qwith  \bb=\displaystyle \sum_{ \emptyset \neq A\subseteq [q] } b_A\be_A\in \mathbb N^{2^q-1}.$$
    By \cref{eq:f-defn}, we have
    \begin{equation}\label{eq:bA}
       b_{A} = \begin{cases}
            2 & \text{if } X \subseteq A,\\  
            0 & \text{if } A\cap X= \emptyset,\\
            1& \text{otherwise.} 
        \end{cases}
    \end{equation}
    By \cref{l:symboliclinalg}\eqref{i:symbolic-equiv}, it suffices to show that if $\C \in \CC_q$ is a minimal covering set of $[q]$, then
    \begin{equation}\label{eq:polyhedron}
        \sum_{A \in \C} \bb_A \geq 2.
    \end{equation}
    Fix $\C \in \CC_q$. If $X \subseteq A$ for some $A \in \C$, then  \cref{eq:polyhedron} holds by \cref{eq:bA}, so we may assume that $X \not\subseteq A$ for all $A \in \C$.  Therefore, since $|X|\geq 2$ and $\C$ covers $[q]$, at least two different sets, say $A_1, A_2 \in \C$, intersect non-trivially with $X$, i.e., $ A_1\cap X\neq \emptyset$ and $A_2\cap X\neq \emptyset$. Thus $b_{A_1}, b_{A_2} \ge 1$, and the inequality in  \cref{eq:polyhedron} follows, concluding the argument that $f_X \in {\E_q}^{(2)}$.   
    
    Next we show that 
      $f_X$ is a {\it minimal} generator of ${\E_q}^{(2)}$. We do so by showing that 
       $$\frac{f_X}{y_A}\notin {\E_q}^{(2)}
         \qforall y_A \in \supp(f_X).$$
    Fix $y_A \in \supp(f_X)$ and
   let $\by^{\bb'}=\frac{f_X}{y_A}$, so $\bb' = \bb - \be_A$. Using  \cref{eq:bA}, there are two scenarios to consider.
   \begin{itemize}
   \item If $A=[q]$, then $X \subseteq A$ and $b_A'=b_A-1=2-1=1$. 
   Since $\{A\}=\{[q]\} \in \CC_q$, using \cref{l:symboliclinalg}\eqref{i:symbolic-equiv}  we conclude that  $\by^{\bb'}\not\in {\E_q}^{(2)}$.
   
   \item If $A \ne [q]$, since $y_A \in \supp(f_X)$, we have $A \cap X \neq  \emptyset$.  By \cref{eq:bA}, if $X\subseteq A$, then $b_A=2$ and $b_{\compl{A}}=0$, and if $X \not \subseteq A$, then since both $A \cap X$ and $\compl{A} \cap X$ are nonempty, then $b_A=b_{\compl{A}}=1$. Hence  
    $$
    b'_{A} + b'_{\compl{A}} = 
    b_{A} - 1 + b_{\compl{A}} = 2 - 1 < 2.
    $$
    Since $\{A,\compl{A}\} \in \CC_q$, by \cref{l:symboliclinalg}\eqref{i:symbolic-equiv} we must have $\by^{\bb'}\not\in {\E_q}^{(2)}$.
  \end{itemize}
Thus $f_X$ is a minimal generator of ${\E_q}^{(2)}$.

\subsubsection*{Proof of $\mingens({\E_q}^{(2)}) \subseteq  \Omega_q^2$} 
     Let $$\by^{\tilde{\bb}} \in \mingens({\E_q}^{(2)}) \qwith  \tilde{\bb}=\displaystyle \sum_{ \emptyset \neq A\subseteq [q] } \tilde{b}_A\be_A\in \mathbb N^{2^q-1} .$$  
     
     Assume $f_{[q]} \mid \by^{\tilde{\bb}}$. Since $f_{[q]} \in \Omega_q^2$ is a minimal generator of ${\E_q}^{(2)}$ by the inclusion above, then $\by^{\tilde{\bb}} =f_{[q]} \in \Omega_q^2$  as desired. 
    Assume from now on that (see \cref{e:smallX})    
     $$f_{[q]}=\by^{\bj+\be_{[q]}} \nmid \by^{\tilde{\bb}}.$$
      It follows that  either $\tilde{b}_{X'} = 0$ for some nonempty proper subset $X'$ of $[q]$, or $\tilde{b}_{[q]}<2$. However, $\tilde{b}_{[q]}=2$ by \cref{l:symboliclinalg}\eqref{i:symbolic-min},   so there exists an $X' \not\in\{\emptyset, [q]\}$ with 
      $\tilde{b}_{X'} = 0$. By \cref{l:symboliclinalg}\eqref{i:symbolic-min} each $\tilde{b}_A \leq 2$, so we set 
      $$
      \DD \coloneqq \{ A\subseteq [q]\colon \tilde{b}_A=0 \} \qand \DD' \coloneqq \{ A\subsetneq [q] \colon  \tilde{b}_A=2  \},$$
      and write
    $$
        \tilde{\bb} = \bj + \be_{[q]} + \sum_{A \in \DD'}\be_A - \sum_{A \in \DD}\be_A.
    $$
        Note that $X' \in \DD$, so $\DD \neq \emptyset$. Since $\by^{\tilde{\bb}} \in \mingens({\E_q}^{(2)})$ by \cref{l:claim}, $\DD$ has a maximum element $X$ with respect to inclusion. Set $f_{\compl{X}} = \by^{\tilde{\bb}'}$. We then have (see \cref{eq:f-defn})
       \begin{align*}
        \tilde{\bb}-\tilde{\bb}' &= 
        \Big( \bj +\be_{[q]} + \sum_{A\in \DD'} \be_A -\sum_{A\in \DD} \be_A   \Big)- 
        \Big(\bj +\be_{[q]} + \sum_{\compl{X}\subseteq A} \be_A -\sum_{A\subseteq X} \be_A  \Big )\\
        &=\Big( \sum_{A\in \DD'} \be_A-\sum_{\compl{X}\subseteq A} \be_A \Big) + 
        \Big(\sum_{A\subseteq X} \be_A-\sum_{A\in \DD} \be_A\Big). 
    \end{align*}
Note that 
$$
 \Big(\sum_{A\subseteq X} \be_A-\sum_{A\in \DD} \be_A\Big) \geq \mathbf{0}
 $$ 
 because if $A\in \DD$, then $A\subseteq X$ since $X$ is a maximum element of $\DD$. 
 For the other difference of sums, if $\compl{X}\subseteq A$, then  by \cref{l:symboliclinalg}\eqref{i:bAT}, since $\tilde{b}_X=0$  we must have  $\tilde{b}_A\geq 2$, which implies that  $A \in \DD'$. Hence  
 $$
 \Big( \sum_{A\in \DD'} \be_A-\sum_{\compl{X}\subseteq A} \be_A \Big)\geq \mathbf{0}.
 $$
Therefore, $\tilde{\bb}-\tilde{\bb}' \geq \mathbf{0}$, and so $f_{\compl{X}} =\by^{\tilde{\bb}'}\divides \by^{\tilde{\bb}}$. By assumption, $\by^{\tilde{\bb}}$ is a minimal generator of ${\E_q}^{(2)}$, and $f_{\compl{X}} \in {\E_q}^{(2)}$ by the first inclusion, so $\by^{\tilde{\bb}} = f_{\compl{X}} \in \Omega_q^2$ as desired.
\end{proof}

We now illustrate how  \cref{thm:gens-2nd-symbolic} can be used to find generators of the second symbolic power of a square-free monomial ideal $I$ using the edge ideal of a triangle. This example also shows that while the generators of ${\E_q}^{(2)}$ have the same degree, this can change under the map $\psi_I$ and images of minimal generators need not be minimal generators.

\begin{example}
Let $q=3$, $X = \{1,2,3\}$, and $Y=\{1,2\}$.  By definition (see \cref{eq:f-defn}),
    \begin{align*}
        f_X &=f_{\{1,2,3\}} = y_{_1} y_{_2} y_{_3} y_{_{12}} y_{_{13}} y_{_{23}} y_{_{123}}^2, \\
        f_Y &= f_{\{1,2\}} = y_{_1} y_{_2} y_{_{12}}^2 y_{_{13}} y_{_{23}} y_{_{123}}^2 = \e_1\e_2 .
    \end{align*}   
    By \cref{thm:gens-2nd-symbolic}, $f_{\{1,2,3\}} \in {\E_3}^{(2)}$ but $f_{\{1,2,3\}} \not \in {\E_3}^{2}$.

    Consider the monomial ideal $I = (x_1 x_2, x_1 x_3, x_2 x_3)$. Then the map $\psi_I$ is given by 
    $$
    \psi_I(y_A) = \begin{cases}
        x_1 &\text{if } A=\{1,2\} \\
        x_2  &\text{if } A=\{1,3\}\\
        x_3  &\text{if } A=\{2,3\} \\
        1 &\text{otherwise}.
    \end{cases}
    $$
    Then $\psi_I(f_{\{1,2,3\}}) = x_1 x_2 x_3 \in I^{(2)}$, but $\psi_I(f_{\{1,2,3\}}) \not \in I^2$ since $\psi_I(f_{\{1,2,3\}})$ is a monomial of degree $3$. By \cref{thm:gens-2nd-symbolic}, $f_{\{1,2\}}$ and $f_{\{1,2,3\}}$ are both minimal generators of ${\E_3}^{(2)}$, however $\psi_I(f_{\{1,2\}})=x_1^2x_2x_3$, which is a multiple of $\psi_I(f_{\{1,2,3\}})$. So while $\psi_I(f_{\{1,2\}}) \in I^{(2)}$, it is no longer a minimal generator.
\end{example}

Note that an alternate approach to determining when $I^{(2)}$ does not have additional generators appears in \cite{RTY} using generalized triangles.

By \cref{c:sdefectbound}, we know that the second symbolic defect of $\E_q$ provides a bound on the symbolic defect of any square-free monomial ideal with $q$ generators. As a consequence of \cref{thm:gens-2nd-symbolic}, we are able to make this bound precise for the second symbolic defect.

\begin{corollary}\label[corollary]{c:explicit-bound}
    For any $q \ge 1$, we have
    \[
    \sdefect(2,\E_q) = 2^q -1-q-\binom{q}{2}.
    \]
\end{corollary}
\begin{proof}    
    By~\cref{thm:gens-2nd-symbolic} we know $\mingens({\E_q}^{(2)})=\Omega_q^2$. By \cref{e:smallX}, $f_X \in {\E_q}^2$ if and only if $|X| \leq 2$. In particular, $\frac{{\E_q}^{(2)}}{{\E_q}^2}$ is minimally generated by $\{f_X \st |X| > 2\}$, hence 
    $$
        \sdefect(2, \E_q) = 2^q - 1 - q - \binom{q}{2},
    $$
    as desired.
\end{proof}

Similar to the scenario for integral closures described in \cref{p:not normal}, we identify elements that are in the symbolic power but not the ordinary power for extremal ideals. These elements will be used later to test the equality of ordinary and symbolic powers for square-free monomial ideals via the map $\psi_I$.  

\begin{proposition}
\label[proposition]{p:not normally torsion free}
    Let $q\ge 3$ and $r \ge 2$ be integers. Then 
    \begin{equation}\label{eq:newhrq}
        h(r,q)=\by^{\bj+\be_{[q]}} \e_1^{r-2} \in {\E_q}^{(r)}\smallsetminus {\E_q}^r.
    \end{equation}
\end{proposition}

\begin{proof}
    By \cref{thm:gens-2nd-symbolic}, $\by^{\bj+\be_{[q]}}=f_{[q]} \in {\E_q}^{(2)}$, and thus for any $r\geq 2$, we have 
    \[
    h(r,q)=\by^{\bj+\be_{[q]}} \e_1^{r-2} \in {\E_q}^{(2)} {\E_q}^{r-2} \subseteq{\E_q}^{(2)} {\E_q}^{(r-2)} \subseteq  {\E_q}^{(r)},\]
    where the containment ${\E_q}^{(a)}{\E_q}^{(b)}\subseteq {\E_q}^{(a+b)}$ is a direct consequence of the definition of symbolic powers of square-free monomial ideals (\cref{d:symbolic-powers}). It remains to show that $h(r,q)\notin {\E_q}^r$. Indeed, if $q\ge 4$, then recall from \cref{p:not normal} that $g(r,q)$, which is a multiple of $h(r,q)$, is not in ${\E_q}^{r}$. It follows  that $h(r,q) \notin {\E_q}^{r}$, as desired. Now assume $q=3$. Suppose  that $h(r,3)\in {\E_3}^r$. Since \[
    \deg h(r,3) = \deg \by^{\bj+\be_{[3]}}+(r-2)\deg \e_1=8+4(r-2)=4r\] 
    and all minimal monomial generators of ${\E_3}^r$ are of degree $4r$, the assumption $h(r,3)\in {\E_3}^r$ in fact implies $h(r,3)\in \mingens({\E_3}^r)$. In other words, we can set $h(r,3)=\e_1^{n_1}\e_2^{n_2}\e_3^{n_3}$ for some non-negative integers $n_1,n_2,n_3$ where $n_1+n_2+n_3=r$. Observe that $h(r,3)=y_{_1}^{n_1}y_{_2}^{n_2}y_{_3}^{n_3}h'$ for some monomial $h'$ that is not divisible by $y_{_1}$, $y_{_2}$, and $y_{_3}$. We then have 
    \[
    y_{_1}^{n_1}y_{_2}^{n_2}y_{_3}^{n_3}h'=h(r,3)= \by^{\bj+\be_{[3]}}\e_1^{r-2}= y_{_1}^{r-1}y_{_2}y_{_3}h''
    \]
    for some monomial $h''$. Since $h'$ is not divisible by $y_{_1}$, $y_{_2}$, and $y_{_3}$, this forces $n_1\geq r-1$ and $n_2,n_3\geq 1$. This implies that $r=n_1+n_2+n_3\geq r+1$, a contradiction. Thus $h(r,3)\not\in {\E_3}^r$.
\end{proof}

We conclude this section by translating the results above regarding properties of $\E_q$ into statements that hold for any ideal $I$ generated by $q$ square-free monomials. Note that this result can be viewed as a symbolic power analog to \cref{t:general-I-ic}. 

\begin{theorem}\label{t:general-I-symbolic}
Let $I$ be an  ideal in the polynomial ring $R$ with $q \ge 1$ minimal square-free monomial generators, let $h(r,q)$ as described in \cref{eq:newhrq}.
The following hold:
\begin{enumerate}
\item\label{i:p-geq-3-symb} if $q \ge 3$, $r\geq 2$, and $\psi_I(h(r,q))\notin I^r$, then $I^{(r)}\ne I^r$;
\item\label{i:q=2-symb}  if $q\le 2$ then $I^{(r)}=I^r$ for all $r \ge 1$;
\item\label{i:h(2,3)-symb} if $q=3$, then $I^{(2)}= I^2+(\psi_I(h(2,3)))$;
\item\label{i:second-symbolic} $I^{(2)}= (\psi_I(f_X)\colon \emptyset\neq X \subseteq [q])$, where $f_X$ is as described in \cref{eq:f-defn};
\item\label{i:second-symb-defect} $\sdefect(2, I) \leq 2^q - 1 - q - \binom{q}{2}$.
\end{enumerate}
\end{theorem}

\begin{proof}
    Since $h(r,q)\in {\E_q}^{(r)}$ by \cref{p:not normally torsion free}, \cref{thm:symbolic-extremal} implies that $\psi_I(h(r,q))\in I^{(r)}$. Therefore Statement (\ref{i:p-geq-3-symb}) follows. 
    
    If $q=2$, Statement (\ref{i:q=2-symb}) is exactly \cref{c:E2-symbolic}. On the other hand, if $q=1$, $\E_1$  is a principal ideal, and thus Statement (2) holds. 
    
    By \cref{eq:mingens-symbolic2}, we have ${\E_3}^{(2)}={\E_3}^2+(f_{[3]})={\E_3}^2+(h(2,3))$. Statement (3) then follows from \cref{thm:symbolic-extremal} and \cref{l:Lr-extremal}(3).

    By \cref{thm:gens-2nd-symbolic}, we have ${\E_q}^{(2)}=  (f_X\colon \emptyset\neq X \subseteq [q])$. Statement (\ref{i:second-symbolic}) then follows by applying the ring homomorphism $\psi_I$ to this equality and using \cref{thm:symbolic-extremal}.
    
    Finally, Statement (\ref{i:second-symb-defect}) follows from \cref{c:sdefectbound,c:explicit-bound}.
\end{proof}

\section{\bf Free resolutions and Betti numbers via $\psi_I$} \label{sec:betti-numbers}
In this section we revisit a result from \cite{Lr} stating that the Betti numbers of $\Erq$ bound the Betti numbers of the powers $I^r$ for any ideal $I$ generated  by $q$ square-free monomials. This result has been further extended in \cite{extremal}, where a multi-graded version is given. The main application is the computation of Betti numbers of $\Erq$, leading to explicit bounds on the Betti numbers of $I^r$, for small values of $q$ or $r$ (\cite{CDFHMS25,extremal}). The results presented here stem from the observation that, in view of the results of this paper, the same methods as in the proof of \cite[Theorem 3.8]{extremal} apply to show that the Betti numbers of $\overline{I^r}$, respectively $I^{(r)}$, are also bounded by the Betti numbers of $\overline{\Erq}$, respectively ${\E_q}^{(r)}$. 

Let $S=\sfk[z_1, \dots, z_s]$ with $\sfk$ a field be a polynomial ring, endowed with a (not necessarily standard) grading. An exact
sequence $F_\bullet$ of free $S$-modules is a {\bf graded free resolution} of a monomial ideal $J$ if it has the form
\begin{equation} \label{graded free}
 0 \to F_d \stackrel{\partial_d}{\longrightarrow} \cdots \to F_i
\stackrel{\partial_i}{\longrightarrow} F_{i-1} \to\cdots \to F_1
\stackrel{\partial_1}{\longrightarrow} F_0
\end{equation}
 where $J \cong
F_0/\im(\partial_1)$, and each map $\partial_i$ is graded, in the
sense that it preserves the degrees of homogeneous elements.
The free resolution in \cref{graded free} is called {\bf minimal} if $\partial_i(F_i) \subseteq (z_1,\ldots,z_s) F_{i-1}$ for every
$i>0$. For each $i\ge 0$, the {\bf Betti number} $\beta_i^S(J)$ is defined as the rank of the free $S$-module $F_i$ when $F_\bullet$ is a minimal resolution of $J$. 

The definition of Betti numbers can be further refined to take into account the grading of $S$. We describe below the notation used in the case when $S$ is endowed with the standard $\mathbb Z^s$-multi-grading. In this case, 
since $\mathbb N^s$ is in one-to-one correspondence with the set of all monomials in $S$, the multi-degree of a monomial $\m$ will be considered to be the monomial itself. Then each $F_i$ can be written as a direct sum of one-dimensional graded free $S$-modules, each of which is denoted $S(\m)$ for some monomial $\m$, indicating the basis is in multi-degree $\m$. Let $\LCM(J)$ denote the lcm-lattice of $J$. Thus, when \cref{graded free} is minimal, 
$$F_i \cong \bigoplus_{\m \in \LCM(J)} S(\m)^{\beta_{i,\m}^S(J)}$$ where the
$\beta_{i,\m}^S(J)$ are the {\bf multigraded Betti numbers} of $J$, which
are invariants of $J$.  

To state our results, we present a general statement that replicates the  idea of \cite[Theorem 3.8]{extremal} and can be specialized to integral closures of powers and symbolic powers. In this more general setting, we consider the ring $R=\sfk[x_1, \dots, x_n]$ and the polynomial ring $S$ as above, and a ring homomorphism  $\psi\colon S\to R$. For our applications, the ring $S$ will be the ring $S_{[q]}$ and the ring map $\psi$ will be the ring homomorphism $\psi_I$, with $I$ a fixed ideal generated by $q$ square-free monomials.

For the remainder of this section, $R$ will be assumed to have the standard $\mathbb Z^n$-multi-grading. We need to equip $S$  with a non-standard multi-grading in order for the map $\psi$ to be homogeneous. We define a non-standard grading on $S$ so that the multi-degree of a monomial $\m$ of $S$ is the same as the standard multi-degree of $\psi(\m)$. An $S$-module $M$ equipped with the standard grading will be regraded as described in \cite{extremal} so that it becomes a graded module with respect to the non-standard multi-grading. This grading can then be induced on graded free resolutions. To distinguish between the two gradings, we write $\widetilde S$, $\widetilde M$, $\widetilde F_\bullet$ for the ring $S$, a module $M$, and a free resolution $F_\bullet$ respectively, considered with the non-standard grading discussed above.

\begin{theorem}[{\bf Extremal ideals bound Betti numbers}]\label{t:upperbound}
 Let $\sfk$ be a field, and let $\psi\colon S\to R$ be a ring homomorphism, where $S$ and $R$ are polynomial rings over $\sfk$. Let $\mathcal J$ denote an ideal in $S$ generated by monomials $u_1, \dots, u_q$, and set $J=\psi(\mathcal J)$. If
 \[
 \psi\left(\lcm(u_{i_1}, \dots, u_{i_t})\right)=\lcm(\psi(u_{i_1}), \dots, \psi(u_{i_t}))
 \]
for all $1\le i_1<i_2<\dots<i_t\le q$, then the following hold: 
  \begin{enumerate}
\item If $F_\bullet$ is a multigraded free resolution of $\mathcal J$ over
  $S$, then $\widetilde F_\bullet\otimes_{\widetilde{S}}R$
  is a multigraded free resolution
  of $J$ over $R$, where in the tensor product $R$ is regarded as
  a graded $\widetilde{S}$-module via the homomorphism $\psi$. 
   
\item  If $\m \in \LCM(J)$, then $$\displaystyle \beta_{i, \m}^R(J) \le
  \sum_{\tiny \begin{array}{ll}\bu\in \LCM(\mathcal J)\\
  \psi(\bu)=\m
  \end{array}}
  \beta_{i,\bu}^S(\mathcal J).$$

  \item If $i\ge 0$ then
      $$\beta_i^R(J)\le \beta_i^S(\mathcal J).$$
  \end{enumerate}
  In particular, if $q$ and $r$ are positive integers, $I$ is an ideal in $R$, and $S=S_{[q]}$, then the statements (1)-(3) hold in each of the following cases
\[
(\mathcal J, J)=(\Erq, I^r)\,; \quad (\mathcal J, J)=(\overline{\Erq}, \overline{I^r})\,;\quad (\mathcal J, J)=({\E_q}^{(r)},I^{(r)})\,.
\]
\end{theorem}
 
\begin{proof}
The proof of (1)-(3) is identical to the proof of \cite[Theorem 3.8]{extremal}, and hence we will not replicate it. The key ingredient in the proof is showing that
\[
\Tor_i^S(S/\mathcal J, R)=0 \qforall i>0.
\]
This is done using the Taylor resolution of a monomial ideal and the fact that the map $\psi$ preserves lcm's. 

Finally, we apply (1)-(3) with $S=S_{[q]}$ and $\psi=\psi_I\colon S_{[q]}\to R$, which satisfy the hypotheses by \cref{l:Lr-extremal}(5). If $\mathcal J=\Erq$ then $J=I^r$ by \cref{l:Lr-extremal}(3). If $\mathcal J=\overline{\Erq}$ then $J=\overline{I^r}$ by \cref{thm:integral-closure-extremal} . If $\mathcal J={\E_q}^{(r)}$ then $J=I^{(r)}$ by \cref{thm:symbolic-extremal}.
\end{proof}

\begin{ack}
The research for this paper was partially supported by the SLMath 2024 Summer Research in Mathematics program. Chau is supported by the Infosys Foundation. Faridi's research is supported by NSERC Discovery Grant 2023-05929. Duval’s research is supported by Simons Collaboration Grant 516801.  The authors are grateful to Karen Smith for useful questions and comments.
\end{ack}

\bibliographystyle{plain}
\bibliography{refs.bib}

\end{document}